\documentclass[11pt,a4paper]{article}
\usepackage[margin=3cm]{geometry}
\usepackage{amsmath,amsthm,stmaryrd, amssymb,verbatim,amsfonts,mathtools,graphicx}
\usepackage[export]{adjustbox}
\usepackage{xcolor,enumerate,verbatim}

\newcommand{\cH}{\mathcal{H}}
\newcommand{\CC}{\mathbb C}\newcommand{\DD}{\mathbb D}
\newcommand{\Cmats}[2]{\CC^{#1{\times}#2}}
\newcommand{\sqmat}[1]{\Cmats{#1}{#1}}
\newcommand{\norm}[1]{\left\| #1 \right\|}

\newcommand{\RHP}{{\mathbb C_+}} 
\newcommand{\RHPtwo}{{\mathbb C_+^2}}
\newcommand{\matr}[1]{\begin{bmatrix}#1\end{bmatrix}}
\newcommand{\re}{\operatorname{Re}}
\newtheorem{thm}{Theorem}[section]
\newtheorem{lemat}[thm]{Lemma}

\newtheorem{remark}[thm]{Remark}

\author{Radomi\l{} Baran \and Piotr Pikul\footnote{PP was supported by National Science Centre, Poland, grant no.~2025/09/X/ST1/00555.} \and Hugo J.\ Woerdeman\footnote{HW is partially supported by NSF grant DMS-2348720.} \and Micha\l{} Wojtylak}
\title{The classes of bivariate Schur and Herglotz  matrix-valued rational functions: realizations, symmetrizations, and related determinantal representations}

\begin{document}
\maketitle

\centerline{\it In memory of Franciszek Hugon Szafraniec (1940--2025)}

\begin{abstract}

We present a finite-dimensional realization theory for bivariate rational functions that are contractive or have nonnegative real part on the bidisc or on the bihalfplane. We show that the realization formula depends only on the underlying domain, while the distinction between the four resulting function classes is captured entirely by explicit matrix inequalities imposed on the realization matrices. These results provide finite-dimensional realizations for rational Schur--Agler and Herglotz--Agler functions, extending the previous infinite-dimensional results. We further characterize symmetric realizations by means of a Hermitian unitary symmetry of the realization data, yielding realization theorems on the symmetrized bihalfplane. Finally, we obtain determinantal representations for symmetric stable polynomials and, consequently, for stable polynomials on the symmetrized bihalfplane.
For rational functions over the real field the respective representations can use matrices with real entries. 
\end{abstract}

\section{Introduction}

 Bivariate, and more generally multivariate, rational matrix functions of this kind arise naturally in the theory of linear parametric systems; see, e.g., \cite{TGN,AIL,AGP}. 
Infinite-dimensional realization theorems for the Schur--Agler and Herglotz--Agler classes, generalizing to several variables the classical one-variable state-space realization theory for Schur- and Herglotz-class functions \cite{dBR1,dBR2,BGR}, were established by Ball and Kaliuzhnyi-Verbovetskyi in \cite{ball2015schur}. By applying a Cayley transform to the unitary  realization of the Schur--Agler class, they obtained a realization  for the Herglotz--Agler class: every function $H$ in the Herglotz--Agler class of the polyhalfplane  was shown there to admit a  realization
\[
H(\zeta) = D + C\bigl(P(\zeta)+A\bigr)^{-1}B, \qquad P(\zeta)=\zeta_1P_1+\dots+\zeta_dP_d,
\]
where $P_1,\dots,P_d$ are orthogonal projections summing to the identity on an auxiliary Hilbert space and the colligation matrix {\tiny $\begin{bmatrix}A&B\\C&D\end{bmatrix}$} satisfies a Kalman--Yakubovich--Popov positivity condition. In the current  paper we study finite-dimensional counterparts of such realizations being is in line with the program of \cite{grinshpan2016}, and continued in \cite{BW,baran2026contractive}. However,  in contrast to these three contributions, the bivariate setting allows us to consider  non-strict matrix inequalities, cf. Figure~\ref{fig:table1}.

Although contractivity and nonnegative real part, each posed on either the bidisc or the bihalfplane, give us four distinct function classes, only two realization formulas are needed between them: the formula depends on the underlying domain, not on which positivity condition is imposed. For functions on the bidisc, the realization has the form
\[
F(z_1,z_2)
 = D+C\bigl(I-(z_1P_1+z_2P_2)A\bigr)^{-1}
   (z_1P_1+z_2P_2)B,
\]
whereas for functions on the bihalfplane it is given by
\[
F(z_1,z_2)
 = D+C(z_1P_1+z_2P_2+A)^{-1}B.
\]
Here $P_1$ and $P_2$ are positive semidefinite matrices satisfying
$
P_1+P_2=I
$.
What separates the four function classes is then a family of matrix inequalities imposed on the realization matrices $A,B,C,$ and $D$.

\begin{figure}[hbt]
\[
\begin{array}{|c||@{}c|c|}
\hline
& \text{contractive} & \text{nonnegative real part}\\
\hline\hline
\rule{0pt}{5.5ex}\mathbb D^2 &
\begin{array}{c}
\displaystyle
\begin{bmatrix}
I-A^*A-C^*C & -A^*B-C^*D\\
-B^*A-D^*C & I-B^*B-D^*D
\end{bmatrix}\ge0
\\[1ex]
\text{Theorem~\ref{thm:kneseDD}}
\end{array}
&
\begin{array}{c}
\displaystyle
\begin{bmatrix}
I-A^*A & -A^*B+C^*\\
-B^*A+C & D+D^*-B^*B
\end{bmatrix}\ge0
\\[1ex]
\text{Theorem~\ref{thm:DDposRe}}
\end{array}
\\
\hline
\rule{0pt}{5.5ex}\RHPtwo &
\begin{array}{c}
\displaystyle
\begin{bmatrix}
A+A^*-C^*C & C^*D+B\\
D^*C+B^* & I-D^*D
\end{bmatrix}\ge0
\\[1ex]
\text{Theorem~\ref{thm:contrRHP}}
\end{array}
&
\begin{array}{c}
\displaystyle
\begin{bmatrix}
A+A^* & C^*-B\\
C-B^* & D+D^*
\end{bmatrix}\ge0
\\[1ex]
\text{Theorem~\ref{thm:posReRHP}}
\end{array}
\\
\hline
\end{array}
\]
\caption{System matrix conditions characterizing contractive and nonnegative real part realizations on $\mathbb D^2$ and $\RHPtwo$.}
\label{fig:table1}
\end{figure}

These conditions are summarized in Figure~\ref{fig:table1}. The upper-left entry is the contractivity condition for the system matrix
{\tiny \(
\begin{bmatrix}
A&B\\
C&D
\end{bmatrix},
\) }
and the lower-right entry is the classical Kalman--Yakubovich--Popov (KYP) inequality. The remaining two entries interpolate between these two cases: the upper-right condition, for instance, combines the $(A,B)$-coefficients of the contractive bidisc realization with the $(C,D)$-coefficients of the nonnegative-real-part bihalfplane realization.

For the upper-left entry of Figure~\ref{fig:table1} -- the matrix valued  Schur class on the bidisc -- a finite-dimensional realization of this kind was already available, going back to Kummert's classical work \cite{Kummert} on lossless two-dimensional scattering and revisited and extended by Knese \cite{knese2021}, based on a Positivstellensatz reasoning;  see Theorem~\ref{thm:kneseDD} below.  We take this bidisc realization as our starting point in Section~\ref{sec:bibi} (Theorem~\ref{thm:kneseDD} below) and derive from it, by finite matrix computations, realizations for the remaining three entries of Figure~\ref{fig:table1}.

A second objective of the paper is to develop symmetric versions of the realization theory, that is the cases when $F(z_1,z_2)= F(z_2,z_1)$. We show that the symmetry condition is equivalent to the existence of a Hermitian unitary matrix $U$ satisfying
\[
\begin{bmatrix}
U&\\
&I
\end{bmatrix}
\begin{bmatrix}
A&B\\
C&D
\end{bmatrix}
\begin{bmatrix}
U&\\
&I
\end{bmatrix}
=
\begin{bmatrix}
A&B\\
C&D
\end{bmatrix},
\qquad
UP_1U=P_2;
\]
see Section~\ref{sec:symmetric}. These results enable us to construct realizations on the symmetrized bihalfplane, paralleling the realization theory on the symmetrized bidisc developed in \cite{AY, BS} and pursued further in \cite{BW}; see Section~\ref{sec:bisym}.

Finally, in Section~\ref{sec:det} we turn the realization theory toward determinantal representations: writing a polynomial with no zeros in a given domain as
\[
p(z_1,z_2) = c\det(A+z_1B_1+z_2B_2)
\]
for suitable matrices $A,B_1,B_2$.  A representation of this kind is in close connection with system realization as
  the term $\det(A+z_1B_1+z_2B_2)$ is the denominator of the transfer-function realizations of Section~\ref{sec:bibi}.
 Recall that on the bihalfplane, Knese \cite{knese2019} obtained such a representation for stable polynomials, with $\operatorname{Im}A\ge0$ and $B_1,B_2\ge0$, $B_1+B_2=I$. 
Our aim is to  extend his result to symmetric stable polynomials and, consequently, to stable polynomials on the symmetrized bihalfplane. 

Notation. We let rational functions over $\mathbb K\in\{\mathbb C,\mathbb R\}$  be denoted by \(\mathbb{K}(z)\), which is the field of fractions of the polynomial ring \(\mathbb{K}[z]\) of one variable. Similar notation follows for bivariate case, we use $\mathbb K[z_1,z_2]$, and $\mathbb K(z_1,z_2)$, respectively. Further,   $\mathbb K^{m\times n}(z_1,z_2)$ stands for  the rational $m\times n$-matrix valued functions of two  variables.
By $\mathbb D$ we denote the (open) unit disc, and by $\RHP$ the (open) right halfplane.

\section{Realizations of the contractive and nonnegative real part functions on  the bidisc and bihalfplane}\label{sec:bibi}

In this section we provide results binding properties
of matrix-valued functions on the bihalfplane with
their system matrix properties, along  the program given by Table~\ref{fig:table1}. We begin 
with the following known  result. 

\begin{thm}
\label{thm:kneseDD}
Let $F\in  \Cmats{m}{n}(z_1,z_2)$ be rational matrix function. Then the following are equivalent:
\begin{enumerate}[\rm (i)]
    \item $F$ has no poles in  $\DD^2$ and 
$\norm{F(z_1,z_2)}\leq 1$ for all $z_1,z_2\in\DD$;
    \item there exists a contractive matrix
$T =\matr{A&B\\C&D}$ such
that for all $z_1,z_2\in \DD$
\begin{equation}\label{S}
F(z_1,z_2)= D+ C(z_1P_1+z_2P_2)\big(I - A(z_1P_1+z_2P_2) \big)^{-1}B
\end{equation}
where $P_1, P_2$ are complementary orthogonal projections.
\end{enumerate} 
Additionally, if $F\in\mathbb R^{m\times n}(z_1,z_2)$ then the matrices $A,B,C,D,P_1,P_2$ in {\rm (ii)} may be chosen as real matrices.
\end{thm}

\begin{proof}
The implication (ii)$\Rightarrow$(i) relies on a well known computation (and also works for $F\in\mathbb C^{m\times n}(z_1,z_2)$); see, for instance, \cite[Proof of Proposition 3.3]{baran2026contractive}. For (i)$\Rightarrow$(ii) 
we refer to \cite[Theorem 1.3]{knese2021}. 
For the statement on $F\in\mathbb R^{m\times n}[z_1,z_2]$ see Theorem~\ref{thm:4vreal} of the Appendix.
    \end{proof}

\begin{remark}\label{Drem}  \rm We note that the full matrix version of \cite[Theorem 1.3]{knese2021} relies on 
 \cite[Theorem 4.1]{Dritschel}, completed only recently  in \cite{dritschel2026}. 
\end{remark}

We will derive now the remaining three statements of Table~\ref{fig:table1}
from the above result.  First we consider contractive functions on the bihalfplane.

\begin{thm}\label{thm:contrRHP}
Let $F\in\Cmats{m}{n}(z_1,z_2)$ be a rational  matrix valued function. Then the following conditions are equivalent:
\begin{enumerate}[\rm (i)]
\item $F$ has no poles in $\RHPtwo$, and $F(z_1,z_2)$ is a contraction for any $z_1,z_2\in\RHP$;
\item $F$ admits a realization of the form
\begin{equation}\label{eq:realization}
F(z_1,z_2)= D+ C(z_1P_1+z_2P_2+A)^{-1}B,
\end{equation}
where the block matrix $\matr{A&B\\C&D} \in\Cmats{(k+m)}{(k+n)}$ is such that
\begin{equation}\label{eq:c1}
\matr{
A+A^*- C^*C & C^*D+ B \\ D^* C+ B^*  & I - D^* D
} \geq 0,
\end{equation}
$P_1,P_2\in \sqmat{k}$ are positive semidefinite and $P_1+P_2=I_k$.
\end{enumerate}
Moreover, if any of the above conditions holds, then
the limit $\lim_{t\to\infty}F(t,t)$ exists.
Additionally, if $F\in\mathbb R^{m\times n }(z_1,z_2)$ then the matrices $A,B,C,D,P_1,P_2$ in {\rm (ii)} may be chosen as real matrices.

\end{thm}

\begin{proof}
We start with showing the `moreover' part. Assuming (ii), the existence of the limit is straightforward, since $\lim_{t\to\infty}(tP_1+tP_2+A)^{-1}=
\lim_{t\to\infty} \frac1t(I+\frac1t A)^{-1} =0$. The case $F\in\mathbb R^{m\times n}[z_1,z_2]$ is deferred to the Appendix, Theorem~\ref{thm:4vreal}.

(ii)$\Rightarrow$(i)
Fix $z=(z_1,z_2)\in \RHP$. We will show that $F(z_1,z_2)$ is a contraction.
From \eqref{eq:c1} we infer that there exist matrices $L\in\Cmats{r}{k}$ and
$M\in\Cmats{r}{n}$ (for some $r\geq 1$) such that

\[ \matr{
A+A^*- C^*C & C^*D+ B \\ D^* C+ B^*  & I - D^* D
} = \matr{L&M}^*\matr{L & M} =\matr{L^*L  & L^*M\\M^*L & M^*M} .\]
Denote $\Delta = z_1P_1+z_2P_2+A$ and observe  that $\Delta+\Delta^* = 2\re(z_1)P_1+2\re(z_2)P_2+2\re(A)=2\re(z_1)P_1+2\re(z_2)P_2+C^*C+L^*L$. Then
\begin{align*}
I-F(z)^*F(z) =\, & I - D^*D- D^*C\Delta^{-1}B- B^*\Delta^{-1}C^* D
 - B^*\Delta^{* -1}C^*C\Delta^{-1}B\\
=\,& M^*M -(M^*L-B^*)\Delta^{-1}B - B^*\Delta^{* -1}(L^*M-B)\\
&\quad +B^*\Delta^{* -1}(L^*L+\re(z_1)P_1+\re(z_2)P_2-\Delta-\Delta^*)\Delta^{-1}B
\\
 =\,&  (M^*-B^*\Delta^{* -1}L^*)(M-L\Delta^{-1}B)\\
& +B^*\Delta^{*-1}\left[\Delta^*+\Delta+\re(z_1)P_1+\re(z_2)P_2-\Delta-\Delta^*
\right]\Delta^{-1}B\\
=\, & (\Delta^{-1}B)^*\big[\re(z_1)P_1+\re(z_2)P_2\big]\Delta^{-1}B\\
&\quad+(M-L\Delta^{-1}B)^*(M-L\Delta^{-1}B).
\end{align*}
Since all summands in the final expression are positive semidefinite,
we have indeed proven that $F(z)$ is a contraction.

(i)$\Rightarrow$(ii)
For $w_1,w_2\in \DD$ define $G(w_1,w_2):=F\big(\frac{1+w_1}{1-w_1},\frac{1+w_2}{1-w_2}\big)$.
Then $G:\DD^2\to \Cmats{m}{n}$ and $\norm{G(w)}\leq 1$ for all $w\in\DD^2$.
By the virtue of Theorem~\ref{thm:kneseDD} we obtain a realization of $G$ as
\begin{equation}
G(w_1,w_2)=D_0+C_0(w_1\tilde P_1+w_2\tilde P_2)\big(I-A_0(w_1\tilde P_1+w_2\tilde P_2)
\big)^{-1}B_0,
\end{equation}
where matrix $T=\matr{A_0&B_0\\C_0& D_0}\in \Cmats{(k+m)}{(k+n)}$ is a contraction and
$\tilde P_1$ and $\tilde P_2$ are complementary orthogonal projections. In particular, $\tilde P_1\tilde P_2=\tilde P_2\tilde P_1=0$.

Decomposing $\CC^k = \cH_1\oplus
\cH_A$ into orthogonal subspaces,
where $\cH_1$ 
is the eigenspace
of $A_0$ for eigenvalue $1$ (as $T$ is a contractive, the eigenspace of $A_0$ at $1$ is a reducing subspace), 
we can see that $T$ can be chosen to take the form
$T=\matr{I&0&0\\0& \tilde A& \tilde C\\ 0& \tilde B& D_0}$,
where  $(I-\tilde A)$ is invertible.
If $1\notin \sigma(A_0)$, we simply have $\cH_1=\{0\}$. 

For $z_1,z_2\in\RHP$ we have
\begin{align*}
F(z_1,z_2)=\,&G\left(\frac{z_1-1}{z_1+1},\frac{z_2-1}{z_2+1}\right)\\
=\,& D_0+C_0\left({\textstyle\frac{z_1-1}{z_1+1}}\tilde P_1
+\textstyle{\frac{z_2-1}{z_2+1}}\tilde P_2\right)
\left(I-A_0({\textstyle\frac{z_1-1}{z_1+1}}\tilde P_1+
{\textstyle\frac{z_2-1}{z_2+1}}\tilde P_2)\right)^{-1}B_0\\
=\,& D_0+C_0 (z_1\tilde P_1+z_2\tilde P_2-I)
\big( (I-A_0)(z_1\tilde P_1+z_2\tilde P_2-I)+2I\big)^{-1}B_0\\
=\,& D_0+\tilde C\left(I-2(I-\tilde A)^{-1}
    \left(z_1 P_1+z_2 P_2-I+2(I-\tilde A)^{-1}\right)^{-1}\right)
    (I-\tilde A)^{-1}\tilde B\\
=\,& D_0{+}\tilde C(I{-}\tilde A)^{-1}\!\tilde B
-2\tilde C(I{-}\tilde A)^{-1}\!
    \left(z_1 P_1{+}z_2 P_2+(I{+}\tilde A)(I{-}\tilde A)^{-1}\right)^{-1}\!
    (I{-}\tilde A)^{-1}\tilde B,
\end{align*}
where $P_j$ is the compression of $\tilde P_j$ to the subspace $\cH_A$ ($j=1,2$).
If we put
\begin{equation}
\begin{array}{ll}
A:=(I+\tilde A)(I-\tilde A)^{-1},&
B:=\sqrt2 (I -\tilde A)^{-1}\tilde B,\\
C:=-\sqrt2 \tilde C (I - \tilde A)^{-1},&
D:= D_0+\tilde C(I - \tilde A)^{-1}\tilde B,
\end{array}
\end{equation}
the formula \eqref{eq:realization} and the assumptions on $P_1$ and $P_2$ hold.

In order to prove \eqref{eq:c1} we introduce two more block matrices
\[ R := \matr{ \frac1{\sqrt2}(A+I) &  \frac1{\sqrt2}B\\0 & I },\quad
L:= \matr{\frac1{\sqrt2}(A-I)  &  \frac1{\sqrt2}B\\-C & D}. \]
If we denote the contractive matrix $\matr{\tilde A&\tilde B\\\tilde C& D_0}$
by $\tilde T$, we have $L=\tilde T R$ and therefore
\[
\matr{
A+A^*- C^*C & C^*D+ B \\ D^* C+ B^*  & I - D^* D
}= R^*R-L^*L 
= R^*(I-\tilde T^*\tilde T)R \geq 0.
\]
This ends the proof.
\end{proof}

Now let us turn to  functions having nonnegative real part  on the
bihalfplane. Recall that their univariate analogue is called frequently `positive real' in  theory of linear systems and univariate transfer functions. The  condition on the system matrix appearing below is
a special case ($Q=I$) of Kalman-Yakubovich-Popov (KYP) inequality, widely used in the univariate case.  

\begin{thm}\label{thm:posReRHP}
Let $F\in\mathbb C^{n\times n}(z_1,z_2)$ be a rational matrix valued function. Then the following conditions are equivalent:
\begin{enumerate}[\rm (i)]
\item {$F$ has no poles in $\RHPtwo$},  $F(z_1,z_2)$ has nonnegative real part
$($i.e.\ $F(z_1,z_2)^*+F(z_1,z_2)\geq 0)$ for any $z_1,z_2\in\RHP$,
and $\lim_{t\to\infty} F(t,t)$ exists;
\item $F$ admits a realization of the form
\begin{equation}\label{eq:realization2} 
F(z_1,z_2)= D+ C(z_1P_1+z_2P_2+A)^{-1}B,
\end{equation}
where the block matrix $\matr{A&B\\C&D} \in\Cmats{(k+n)}{(k+n)}$ is such that
\begin{equation}\label{eq:kyp0}
\matr{
A+A^*  & C^*- B \\ C- B^* & D^*+ D
} \geq 0,
\end{equation}
$P_1,P_2\in \sqmat{k}$ are positive semidefinite and $P_1+P_2=I_k$.
\end{enumerate}
Additionally, if $F\in\mathbb R^{n\times n}(z_1,z_2)$ then the matrices $A,B,C,D,P_1,P_2$ in {\rm (ii)} may be chosen as real matrices.
\end{thm}

\begin{proof}
(ii)$\Rightarrow$(i)
The fact that $\lim_{t\to\infty} F(t,t)=D$ follows the same argument
as in the proof of Theorem~\ref{thm:contrRHP}.

Fix $z=(z_1,z_2)\in \RHP$. We will show that $F(z)^*+F(z)\geq 0$.
From \eqref{eq:kyp0} we infer that there exist matrices $L\in\Cmats{r}{k}$ and
$M\in\Cmats{r}{n}$ (for some $r\geq 1$) such that 
\[ \matr{
A+A^*  & C^*- B \\ C- B^* & D^*+ D
} = \matr{L&M}^*\matr{L & M} =\matr{L^*L  & L^*M\\M^*L & M^*M} .\]
Denote $\Delta = z_1P_1+z_2P_2+A$ and observe  that $\Delta+\Delta^* = 2\re(z_1)P_1+2\re(z_2)P_2+2\re(A)=2\re(z_1)P_1+2\re(z_2)P_2+L^*L$. Then
\begin{align*}
F(z)^*{+}F(z) =\, & D^*+D+C\Delta^{-1}B+B^*\Delta^{-1}C^*\\
=\,& M^*M+ (M^*L+B^*)\Delta^{-1}B + B^*\Delta^{* -1}(L^*M+B)\\
 =\,&  (M^*+B^*\Delta^{* -1}L^*)(M+L\Delta^{-1}B)+
 B^*(\Delta^{*-1}{+}\Delta^{-1}-\Delta^{* -1}L^* L\Delta^{-1}
 )B \\
=\, & (\Delta^{-1}B)^*\big[\re(z_1)P_1+\re(z_2)P_2\big]\Delta^{-1}B 
+(M+L\Delta^{-1}B)^*(M+L\Delta^{-1}B).
\end{align*}
Since all summands in the final expression are positive semidefinite,
we have indeed proven that $F(z)$ has nonnegative real part.

\noindent (i)$\Rightarrow$(ii)
Define $G(z):=(F(z)-I)(F(z)+I)^{-1}$ for $z\in\RHPtwo$. Since $\re F(z)\geq 0$
this is a well defined rational function on $\RHPtwo$
attaining contractive values {(scalars in $\DD$)}.
According to Theorem~\ref{thm:contrRHP}, there exist realization of $G$ in the form
\[ G(z_1,z_2)= \tilde D+ \tilde C\big(z_1P_1+z_2P_2+\tilde A\big)^{-1} \tilde B.\]

Then $G(z)= -2I(F(z)+I)^{-1} + I$, so $(G(z)-I)(F(z)+I) =2I$. Since
$\tilde D=\lim_{t\to\infty}G(t,t)$, taking $z=(t,t)$ and letting $t\to\infty$
we find that $\tilde D-I$ is invertible since $(\tilde D-I)(\tilde D+I)=2I$.
From the definition of $G$ we derive
\[ F(z_1,z_2)=-2(G(z_1,z_2)-I)^{-1}-I
= -I+2\left( I-\tilde D- \tilde C(z_1P_1+z_2P_2+\tilde A)^{-1}\tilde B\right)^{-1}\]
and using the Woodbury matrix identity
$(U-QR^{-1}S)^{-1}=U^{-1}+U^{-1}Q(R-SU^{-1}Q)^{-1}SU^{-1}$
(see e.g.\ \cite{Gut1946}) we obtain the realization formula
\[ F(z_1,z_2)= D + C (z_1P_1+z_2P_2+A)^{-1}B,\]
where \newcommand{\ImD}[1][]{(I-\tilde D #1)^{-1}}
\[\begin{array}{ll}A:=\tilde A-\tilde B\ImD\tilde C,&
B:= \sqrt 2\tilde B\ImD,\\
C:=\sqrt 2\ImD\tilde C,&
D:=2\ImD-I = (I+\tilde D)\ImD.\end{array}\]
It remains to show inequality \eqref{eq:kyp0}. Multiplying by
$I \oplus \frac{I-\tilde D^*}{\sqrt2}$ and $I \oplus \frac{I-\tilde D}{\sqrt2}$
from the left and right side respectively, we obtain an equivalent inequality
\begin{equation}\label{eq9}
\matr{ \tilde A+\tilde A^*- \tilde B(I{-}\tilde D)^{-1}\tilde C-
    \tilde C^*(I{-}\tilde D^*)^{-1}\tilde B^* & \tilde C^*(I{-}\tilde D^*)^{-1}(I{-}\tilde D)- B \\ (I{-}\tilde D^*)(I{-}\tilde D)^{-1} \tilde C- \tilde B^*
 & I - \tilde D^* \tilde D
} \ge 0.
\end{equation}

Let us first assume that $\| \tilde D\|<1.$ In that case we can take a Schur complement in \eqref{eq9} and obtain 
\begin{equation}\label{eq:shurT}
\begin{array}{c}
\tilde A+\tilde A^*- \tilde B(I-\tilde D)^{-1}\tilde C-\tilde C^*(I-\tilde D^*)^{-1}\tilde B^* -  \\
\big(\tilde C^*(I-\tilde D^*)^{-1}(I-\tilde D)- B\big)(I-\tilde D^*\tilde D)^{-1}\big((I-\tilde D^*)(I-\tilde D)^{-1}\tilde C- \tilde B^*\big).\end{array}\end{equation}
Taking the Schur complement in \eqref{eq:c1} we obtain
\begin{equation}\label{eq:shur2}
\tilde A+\tilde A^*- \tilde C^*\tilde C - (\tilde C^*\tilde D+ \tilde B)(I-\tilde D^* \tilde D)^{-1}(\tilde D^*\tilde C+\tilde B^*).
\end{equation} 
Upon expanding \eqref{eq:shurT} and \eqref{eq:shur2} we observe
their equality, thanks to the following identities:
\begin{align*}
\tilde C^*(I-\tilde D^*)^{-1}&(I-\tilde D) 
(I-\tilde D^*\tilde D)^{-1}(I-\tilde D^*)^{-1}(I-\tilde D)\tilde C
= \tilde C^*(I-\tilde D\tilde D^*)^{-1}\tilde C\tag{*}\label{eq:t5*}\\
&= \tilde C^*\tilde C+\tilde C^*\tilde D(I-\tilde D^*\tilde D)^{-1}\tilde D^*\tilde C,
\\
\tilde B(I-\tilde D)^{-1}\tilde C-& 
\tilde B(I-\tilde D^*\tilde D)^{-1}(I-\tilde D^*)(I-\tilde D)^{-1}\tilde C
\tag{**}\\
&=\tilde B\big(I-(I-\tilde D^*\tilde D)^{-1}(I-\tilde D^*)\big)(I-\tilde D)^{-1}\tilde C
= \tilde B(I-\tilde D^*\tilde D)^{-1}\tilde D^*\tilde C.
\end{align*}

The proof of equality \eqref{eq:t5*} relies on
the  Woodbury matrix identity and its particular consequence
$D(I-D^*D)^{-1}=(I-DD^*)^{-1}D$.
Having the equality of the Schur complements we infer  
from \eqref{eq:c1} that \eqref{eq9} holds when $\|\tilde D \|<1$.

Finally, if $\|\tilde D\| =1$ we can replace $\tilde B$ and $\tilde D$ by $r\tilde B$ and $r\tilde D$, where $r< 1$. For such $\tilde  B$ and $\tilde D$ condition \eqref{eq:c1} is satisfied and therefore
\eqref{eq9} holds. Subsequently, let $r\to 1$ from below to obtain \eqref{eq9}.

The case $F\in\mathbb R^{n\times n}[z_1,z_2]$ is deferred to the Appendix,  Theorem~\ref{thm:4vreal}.
\end{proof}

Finally, we present a result on   functions having nonnegative real part on the bidisc.

\begin{thm}\label{thm:DDposRe}
Let $F\in{\mathbb C}^{n\times n}(z_1,z_2)$ be a rational function. Then the following conditions are equivalent:
\begin{enumerate}[\rm (i)]
\item $F$ has no poles in $\mathbb D^2$, and the matrix $F(z_1,z_2)$ has nonnegative real part for any $z_1,z_2\in\DD$;
\item $F$ admits a realization of the form
\begin{equation}\label{eq:rDisc}
F(z_1,z_2)= D+ C(z_1P_1+z_2P_2)\big(I - A(z_1P_1+z_2P_2) \big)^{-1}B
\end{equation}
where the block matrix $\matr{A&B\\C&D} \in\Cmats{(k+n)}{(k+n)}$ is such that
\begin{equation}\label{eq:c6}
\matr{
I - A^* A & -A^* B+ C^* \\ -B^*A+ C & D+D^*- B^*B
} \geq 0,
\end{equation}
$P_1,P_2\in \sqmat{k}$ are positive semidefinite and $P_1+P_2=I_k$.
\end{enumerate}
Additionally, if $F\in\mathbb R^{n\times n}(z_1,z_2)$ then the matrices $A,B,C,D,P_1,P_2$ in {\rm (ii)} may be chosen as real matrices.
\end{thm}

\begin{proof}
    (ii) $\Rightarrow$ (i): For convenience, let us write $Z=z_1P_1+z_2P_2$, which is a contraction when $(z_1,z_2)\in {\mathbb D}^2$. In addition, write
    \begin{equation}\label{eq:c66}
\matr{
I - A^* A & -A^* B+ C^* \\ -B^*A+ C & D+D^*- B^*B
} = \matr{ L^* \\ M^* } \matr{ L & M }. 
\end{equation}
Now 
$$ F(z)+F(z)^* = D+D^* + C(I-ZA)^{-1}ZB + B^*Z^*(I-A^*Z^*)^{-1}C^*.$$ Using $D+D^*= M^*M+B^*B$, $C=B^*A+M^*L$, and rules like $(I-ZA)^{-1}Z=Z(I-AZ)^{-1}$, we get
$$F(z)+F(z)^*= M^*M+B^*B + (B^*A+M^*L)Z(I-AZ)^{-1}B+ B^*(I-AZ)^{*-1}Z^*(A^*B+L^*M)$$
$$ = (M^* + B^*(I-AZ)^{*-1}Z^*L^*)(M+LZ(I-AZ)^{-1}B) + $$
$$ B^*(I-AZ)^{*-1}[(I-Z^*A^*)(I-AZ) + (I-Z^*A^*)AZ + Z^*A^*(I-AZ)-Z^*L^*LZ](I-AZ)^{-1}B$$
$$\ge B^*(I-AZ)^{*-1} [I-Z^*Z](I-AZ)^{-1}B \ge 0$$
for $(z_1,z_2)\in {\mathbb D}^2$, where we used that $L^*L=I-A^*A$.

(i) $\Rightarrow$ (ii):
Define $G(z):=(F(z)-I)(F(z)+I)^{-1}$ for $z\in{\mathbb D}^2$. Since $\re F(z)\geq 0$
this is a well defined rational function on ${\mathbb D}^2$ attaining contractive values.
According to Theorem~\ref{thm:kneseDD}, there exist realization of $G$ in the form
\[ G(z_1,z_2)= \tilde D+ \tilde C(z_1P_1+z_2P_2) \big(I- \tilde A(z_1P_1+z_2P_2)\big)^{-1} \tilde B\]
with $$ \Gamma:= \matr{\tilde A & \tilde B \\ \tilde C & \tilde D }  $$ a contraction and $P_1, P_2$ complementary orthogonal projections. When 1 is an eigenvalue of $\tilde D$, we have that its eigenspace is a reducing subspace, and we obtain that $\Gamma$ takes the form 
$$ \Gamma = \matr{\tilde A & \hat B \\ \hat C & \hat D } \oplus I . $$
This translates into $F(z)$ being of the form $\hat F (z) \oplus 0$. Thus, without loss of generality, we may assume that 1 is not an eigenvalue of $\tilde D$, i.e., $I-\tilde D$ is invertible. 

Performing the same calculations as in the proof of Theorem \ref{thm:posReRHP}, we obtain that
\[ F(z_1,z_2)=-2(G(z_1,z_2)-I)^{-1}-I
= 
 D + C(I-(z_1P_1+z_2P_2)A)^{-1}(z_1P_1+z_2P_2)B,\]
where \newcommand{\ImD}[1][]{(I-\tilde D #1)^{-1}}
\[\begin{array}{ll}A:=\tilde A+\tilde B\ImD\tilde C,&
B:= \sqrt 2\tilde B\ImD,\\
C:=\sqrt 2\ImD\tilde C,&
D:=2\ImD-I = (I+\tilde D)\ImD.\end{array}\]
 It remains to show inequality \eqref{eq:c6}. 
Multiplying the left hand side of \eqref{eq:c6} by
$I \oplus \frac{I-\tilde D^*}{\sqrt2}$ and $I \oplus \frac{I-\tilde D}{\sqrt2}$
from the left and right side respectively, a computation shows that it equals 
$$ \matr{L^*+\tilde C^* (I-\tilde D)^{*-1} M^* \cr M^*}\matr{L+M (I-\tilde D)^{-1} C &  M}, $$
where $L$ and $M$ are defined via
$$ \matr{I & 0 \cr 0 & I } - \matr{\tilde A^* & \tilde C^*\cr \tilde B^* & \tilde D^*} \matr{\tilde A & \tilde B \cr \tilde C &  \tilde D } = \matr{L^* \cr M^*} \matr{L & M}. $$ This proves \eqref{eq:c6}.

The case $F\in\mathbb R^{n\times n}[z_1,z_2]$ is deferred to the Appendix,  Theorem~\ref{thm:4vreal}.
\end{proof}

\section{Realizations of symmetric rational matrix functions}\label{sec:symmetric}

Now we present a symmetrized version of Theorems \ref{thm:kneseDD}, \ref{thm:contrRHP}, \ref{thm:posReRHP}, and \ref{thm:DDposRe}.
We call a matrix-valued function $F\in\mathbb C^{m\times n}(z_1,z_2)$ \emph{symmetric} if
$F(z_1,z_2)=F(z_2,z_1)$  whenever $(z_1,z_2)$ is not a singularity of $F$.

Characterization of symmetric contractive-valued functions on a bidisk
was proven in \cite{BW}. We provide a slightly modified statement to align with the forthcoming results.
\begin{thm}\label{thm:contrDDsym}
Let $F:\DD^2\to\Cmats{m}{n}$ be a rational matrix function satisfying
Theorem~\ref{thm:kneseDD}~{\rm(i)}. Then the following conditions are equivalent:
\begin{enumerate}[\rm(i)]
\item function $F$ is symmetric;
\item the matrices $A,B,C,D,P_1,P_2$ from Theorem~\ref{thm:kneseDD} {\rm (ii)}
can be chosen such that
\begin{equation}\label{eq:symre}
 CU=C,\quad UB=B,\quad UAU=A, \quad UP_1U=P_2, 
\end{equation}
for some $U\in\Cmats{k}{k}$  satisfyig $U=U^*=U^{-1}$.
\end{enumerate}
Additionally, if $F\in\mathbb R^{m\times n}(z_1,z_2)$ then the matrices $A,B,C,D,P_1,P_2,U$ in {\rm (ii)} may be chosen as real matrices.
\end{thm}

\begin{proof}
The   implication  (ii)$\Rightarrow$(i)  is a matter of straightforward computation:
\begin{align*}
F(z_2,z_1)=\,& D+C(z_2P_1+z_1P_2+A)^{-1}B= D+CU(z_2UP_1U+z_1UP_2U+UAU)^{-1}UB\\
=\,& D+C(z_2P_2+z_1P_1+A)^{-1}B= F(z_1,z_2).
\end{align*} 

(i)$\Rightarrow$(ii) follows by taking the $A_1,A_2,B,C,D$ as in Theorem 2.1 of \cite{BW} and setting 
\[
\tilde A:=\matr{A_1&A_2\\A_2&A_1},\quad 
\tilde B:=  \matr{B\\B},\quad  \tilde C:= \matr{C&C},\quad 
P_1:=\matr{I&0\\0&0},\quad P_2:=\matr{0&0\\0&I}.\]
The matrices $\tilde A,\tilde B,\tilde C, D$, and $P_1,P_2$ form a realization satisfying (ii), with the unitary involution  $U=\matr{0&I\\I&0}$.

The case $F\in\mathbb R^{m\times n}(z_1,z_2)$ follows by inspection. 
\end{proof}

A symmetric version of Theorem~\ref{thm:contrRHP} is the following:

\begin{thm}\label{thm:ScontRHP}
Let $F:\RHPtwo\to \Cmats{m}{n}$ be a rational matrix function such that
$\norm{F(z)}\leq 1$ for all $z\in\RHPtwo$. Then the following conditions are equivalent:
\begin{enumerate}[\rm (i)]
\item function $F$ is symmetric;
\item $F$ admits a realization of the form \eqref{eq:realization} such that \eqref{eq:c1} and
\eqref{eq:symre} 
hold for some unitary involution $U\in\Cmats{k}{k}$.
\end{enumerate}
Additionally, if $F\in\mathbb R^{m\times n}(z_1,z_2)$ then the matrices $A,B,C,D,P_1,P_2,U$ in {\rm (ii)} may be chosen as real matrices.
\end{thm}
\begin{proof}
The proof of (ii)$\Rightarrow$(i) is identical to the analogous part
for Theorem~\ref{thm:contrDDsym}.

(i)$\Rightarrow$(ii) 
According to Theorem~\ref{thm:contrRHP}, function $F$ admits a realization
\eqref{eq:realization} with matrix $\matr{\tilde A&\tilde B\\\tilde C&\tilde D}
\in\Cmats{(\tilde k+m)}{(\tilde k+1)}$ and positive semidefinite maps
$\tilde P_1,\tilde P_2\in\sqmat{\tilde k}$.
By symmetry we have
\begin{align*}
F(z_1,z_2)=\,&\frac{F(z_1,z_2)+F(z_2,z_1)}{2}=
\tilde D+\tilde C\frac{(z_1\tilde P_1+z_2\tilde P_2+\tilde A)^{-1}+
(z_1\tilde P_2+z_2\tilde P_1+\tilde A)^{-1}}{2}\tilde B\\
=\,& \tilde D+\frac12 \matr{\tilde C&\tilde C}
    \matr{z_1\tilde P_1+z_2\tilde P_2+\tilde A&0\\
    0&z_1\tilde P_2+z_2\tilde P_1+\tilde A}^{-1}\matr{\tilde B\\\tilde B}\\
=\,& \tilde D+\frac12\matr{\tilde C&\tilde C}\left(z_1\matr{\tilde P_1&\\&\tilde P_2}
    +z_2\matr{\tilde P_2&\\&\tilde P_1}
    +\matr{\tilde A&\\&\tilde A}\right)^{-1}\matr{\tilde B\\\tilde B}.
\end{align*}
Then taking $A:=\matr{\tilde A&\\&\tilde A}$,
$B:=\frac{1}{\sqrt{2}}\matr{\tilde B\\\tilde B}$,
$C:=\frac{1}{\sqrt{2}}\matr{\tilde C&\tilde C}$, $D:= \tilde D$,
$P_1:=\matr{\tilde P_1&\\&\tilde P_2}$, $P_2:=\matr{\tilde P_2&\\&\tilde P_1}$
and $U:=\matr{0&I_{\tilde k}\\I_{\tilde k}&0}$ suffices to satisfy \eqref{eq:symre}.

It remains to show that the new realization satisfies \eqref{eq:c1}, i.e.\ that the matrix
\[T = \matr{
A+ A^*- C^* C  & C^*D+  B \\ D^*  C+  B^*
 & I -  D^* D
} = \matr{
\tilde A+\tilde A^*- \frac12 \tilde C^*\tilde C & - \frac12 \tilde C^*\tilde C & \frac1{\sqrt2}(\tilde C^* \tilde D+ \tilde B) \\
- \frac12 \tilde C^*\tilde C & \tilde A+\tilde A^*- \frac12 \tilde C^* \tilde C & \frac1{\sqrt2}(\tilde C^* \tilde D+ \tilde B) \\
\frac1{\sqrt2}(\tilde D^*\tilde C+ \tilde B^*)  & \frac1{\sqrt2}(\tilde D^*\tilde C+ \tilde B^*) &  I - \tilde D^* \tilde D }
\]
is positive semidefinite. Introducing
\[ \tilde T= \matr{ \tilde A+\tilde A^*-\tilde C^*\tilde C & \tilde C^*\tilde D+\tilde B \\ \tilde D^*\tilde C+\tilde B^*
    & I-\tilde D^*\tilde D }
\text{ and }
V=\matr{\frac1{\sqrt2}I_{\tilde k}&\frac1{\sqrt2}I_{\tilde k}&0\\-\frac1{\sqrt2}I_{\tilde k}&\frac1{\sqrt2}I_{\tilde k}&0\\
0&0&I
},
\]
allows us to write
\[ V^*T V = \matr{\tilde A+\tilde A^* & 0 & 0 \\0 & \tilde A+\tilde A^*-{\tilde C}^*\tilde C & \tilde C^*\tilde D+\tilde B \\ 0 &  \tilde D^*\tilde C+{\tilde B}^*
    &I-\tilde D^*\tilde D } = (\tilde A+\tilde A^*) \oplus \tilde T,\quad \]
where the positive semidefiniteness of $\tilde T$ and $\tilde A+ \tilde A^*$ is a consequence of \eqref{eq:c1}.

The case $F\in\mathbb R[z_1,z_2]$ follows by inspection. 
\end{proof}

In our study of symmetric functions the most important case are those having nonnegative real part
on the bihalfplane, addressed in the following theorem.

\begin{thm}\label{thm:posReRHPsym}
Let $F\in\mathbb C^{n\times n}(z_1,z_2)$ be a rational matrix function as in Theorem~\ref{thm:posReRHP}~{\rm(i)}.  
Then the following conditions are equivalent:
\begin{enumerate}[\rm (i)]
\item function $F$ is symmetric;
\item the matrices $A,B,C,D,P_1,P_2$ from Theorem~\ref{thm:posReRHP}~{\rm(ii)} can be chosen in such a way that \eqref{eq:symre} is satisfied
for some unitary involution $U\in{\mathbb C}^{k\times k}$.
\end{enumerate}

Additionally, if $F\in\mathbb R^{n\times n}(z_1,z_2)$ then the matrices $A,B,C,D,P_1,P_2,U$ in {\rm (ii)} may be chosen as real matrices.
\end{thm}

\begin{proof}
The implication (ii)$\Rightarrow$(i) follows from the same computation as in the proof of Theorem~\ref{thm:contrDDsym}.

For the converse statement, 
we construct $A$, $B$, $C$, $D$, $P_1$,
and $P_2$ in the same way as in the proof of Theorem~\ref{thm:ScontRHP}.
Then it remains to prove that 
\[ T =\matr{A+A^*&C^*-B\\C-B^*&D+D^*}= \matr{\tilde A+\tilde A^*& 0 & \frac{\tilde C^*-\tilde B}{\sqrt2}\\
0 & \tilde A+\tilde A^*& \frac{\tilde C^*-\tilde B}{\sqrt2}\\
\frac{\tilde C-\tilde B^*}{\sqrt{2}}&\frac{\tilde C-\tilde B^*}{\sqrt{2}}&
\tilde D+\tilde D^*}\geq 0.\]
It is straightforward, for $T=\matr{1&0\\0&0\\0&\frac1{\sqrt2}}\tilde T \matr{1&0&0\\0&0&\frac1{\sqrt2}}
+ \matr{0&0\\1&0\\0&\frac1{\sqrt2}} \tilde T \matr{0&1&0\\0&0&\frac1{\sqrt2}}$, where $\tilde T$ is
the positive semidefinite matrix from \eqref{eq:kyp0} for the realization
$\matr{\tilde A&\tilde B\\\tilde C&\tilde D}$.

The case $F\in\mathbb R[z_1,z_2]$ follows by inspection. 
\end{proof}

There is a symmetric counterpart to Theorem~\ref{thm:DDposRe} left. The proof follows the same steps as the proof of Theorem~\ref{thm:posReRHPsym}.
\begin{thm}\label{thm:DDposReSym}
Let $F\in\mathbb C^{n\times n}(z_1,z_2)$ be a rational matrix function as in Theorem~\ref{thm:DDposRe}~{\rm(i)}.  
Then the following conditions are equivalent:
\begin{enumerate}[\rm (i)]
\item function $F$ is symmetric;
\item the matrices $A,B,C,D,P_1,P_2$ from Theorem~\ref{thm:DDposRe}~{\rm(ii)} can be chosen
in such a way that additionally \eqref{eq:symre} holds.
\end{enumerate}
Additionally, if $F\in\mathbb R^{n\times n}(z_1,z_2)$ then the matrices $A,B,C,D,P_1,P_2,U$ in {\rm (ii)} may be chosen as real matrices.
\end{thm}

\section{Realizations on the symmetrized bidisk and bihalfplane}\label{sec:bisym}

We define two symmetrized domains of $\mathbb C^2$. First one is the symmetrized bidisk 
\(\mathbb{G}:=\{(s,p)\in \mathbb{C}^{2}:s=z+\zeta,p=z\zeta\text{\ for\ some\ }z,\zeta\in \mathbb{D}\}\).
Analogously, we define the symmetrized bihalfplane
\[
\mathbb{GH} = \left\{(s, p) \in {\mathbb C}^2 : \exists\, z, \zeta \in \RHP,\ s = z + \zeta,\ p = z \zeta \right\}.
\]
Note that it is the image of $\RHPtwo$  under the symmetrization map $(z, \zeta) \mapsto (z + \zeta, z \zeta)$.
By convention, the variables of these symmetrized domains are denoted by $(s,p)$, rather than $(z_1,z_2)$. 

In \cite{BW} finite dimensional realizations are established for contractive symmetric rational functions on the bidisk $\mathbb{D}^2$ and contractive rational functions on the symmetrized bidisk $\mathbb{G}$. In this section we provide realization results for contractive functions and nonnegative real part functions on $\mathbb{GH}$ as well as a realization result for nonnegative real part functions on $\mathbb{G}$. 

\begin{thm}\label{thm:contrGH}
Let $G\in \mathbb C^{m \times n}(s,p)$ be a rational matrix valued  function. Then the following conditions are equivalent:
\begin{enumerate}[\rm (i)]
\item $G$ has no singularities in $\mathbb{GH}$ and the matrix $G(s,p)$ has norm less than or equal to one
for any $(s,p)\in \mathbb{GH}$;
\item $G$ admits a realization of the form 
    $$G(s,p)=  D+  \frac12 {C} \left(s I+ 2A -
    (s^2-4p) P(s I+ 2A)^{-1} P\right)^{-1}B. $$
    with 
    \begin{equation}\label{eq:c1a}
\matr{
A+A^*- C^*C  & C^*D+ B\\ D^* C+ B^*  & 
I - D^* D} \geq 0,
\end{equation}
   and $-I\le P \le I$.
   \end{enumerate}
Additionally, if $G\in {\mathbb R}^{m\times n}(s,p)$, then the matrices $A,B,C,D$ and $P$
can be chosen as real.
\end{thm}

\begin{proof}
Given $G(s,p)$ as stated, we let $F(z,\zeta) = G(z+\zeta, z\zeta).$ Then $F$ satisfies the conditions of Theorem~\ref{thm:ScontRHP} and thus we may write $F$ as in its proof:
$$ F(z,\zeta)=  D+\frac12\matr{ C& C}\left(z\matr{ P_1&\\& P_2}
+\zeta\matr{ P_2&\\& P_1}
+\matr{ A&\\& A}\right)^{-1}\matr{ B\\ B}.$$
Inserting ${\tiny U=\begin{bmatrix} \frac{1}{\sqrt{2}} I & \frac{1}{\sqrt{2}} I \cr \frac{1}{\sqrt{2}} I & -\frac{1}{\sqrt{2}} I  \end{bmatrix}}$
in various places, we may rewrite this as
\[F(z,\zeta)=  D+\frac12\matr{ \frac{C}{\sqrt{2}}& 0 }\matr{\frac12 (z+\zeta) ( P_1+  P_2)+ A&\frac12 (z-\zeta) ( P_1-  P_2)\\
\frac12(z-\zeta) ( P_1-  P_2)&\frac12(z+\zeta) ( P_1+  P_2)+ A}
^{-1}\matr{ \frac{B}{\sqrt{2}}\\0}=\]
$$ D+ \frac12 C \Bigl[(z{+}\zeta) ( P_1{+}P_2)+ 2A -
{(z{-}\zeta)^2} ( P_1{-} P_2)\bigl[(z{+}\zeta) ( P_1{+}  P_2)+ 2A\bigr]^{-1} ( P_1{-} P_2)\Bigr]^{-1}B.$$
Let us put $s=z+\zeta, p=z\zeta$, and observe that $(z-\zeta)^2= s^2-4p$. Note that $ P_1+  P_2 =I$.
We then get
$$G(s,p)=  D+  \frac12{C} \left(s I+ 2A -
    (s^2-4p) ( P_1-  P_2)(s I+ 2A)^{-1} ( P_1-  P_2)\right)^{-1}B. $$
    Lastly, we let $P=P_1-P_2$, giving the desired realization of $G$.

For the converse, starting with the realization of $G$, and putting $P_1=\frac12(I+P)$, $P_2=\frac12(I-P)$, one can reverse the calculation above and conclude that via Theorem \ref{thm:ScontRHP} that $F$ takes on contractive values, but then so does $G$.

The case $G\in {\mathbb R}^{m\times n}(s,p)$ follows from the corresponding part of Theorem~\ref{thm:ScontRHP} by observing that all the steps below can be done over ${\mathbb R}$.

\end{proof}

\begin{thm}\label{thm:posReRHP2a}
Let $G\in \mathbb C^{n\times n} (s,p)$ be a rational  matrix  function. Then the following conditions are equivalent:
\begin{enumerate}[\rm (i)]
\item $G$ has no singularities in $\mathbb{GH}$, the matrix $G(s,p)$ has nonnegative real part
for any $(s,p)\in \mathbb{GH}$,
and $\lim_{(s,p)\to\infty} G(s,p)$ exists;
\item $G$ admits a realization of the form
\begin{equation}\label{eq:realization3} 
G(s,p)=  D+  \frac12{C} \left(s I+ 2A -
    (s^2-4p) P(s I+ 2A)^{-1} P\right)^{-1}B.
\end{equation}
where the block matrix $\matr{A&B\\C&D} \in\Cmats{(k+n)}{(k+n)}$ is such that
\begin{equation*}
\matr{
A+A^* & C^*- B \\ C- B^* & D+D^*
} \geq 0
\end{equation*}
and $-I\le P \le I$.
\end{enumerate}
Additionally, if $G\in {\mathbb R}^{n\times n}(s,p)$, then the matrices $A,B,C,D$ and $P$ can be chosen as real.
\end{thm}

\begin{proof} 

Given $G(s,p)$ as stated, we let $F(z,\zeta) = G(z+\zeta, z\zeta).$ Then $F$ satisfies the conditions of Theorem \ref{thm:posReRHP} and thus we may write $F$ as in its proof:
$$ F(z,\zeta)=  D+\frac12\matr{ C& C}\left(z\matr{ P_1&\\& P_2}
    +\zeta\matr{ P_2&\\& P_1}
    +\matr{ A&\\& A}\right)^{-1}\matr{ B\\ B}.$$
Performing the same calculations as in the proof of Theorem \ref{thm:contrGH}, we arrive at \eqref{eq:realization3}.

(ii)$\Rightarrow$(i): Reversing the computations above, we get that $f$ satisfies the conditions of Theorem \ref{thm:posReRHP}.

The case $G\in {\mathbb R}^{n\times n}(s,p)$ follows from the corresponding part of Theorem~\ref{thm:posReRHP}  by observing that all the steps above can be done over ${\mathbb R}$.
\end{proof}

\begin{thm}Let $G\in \mathbb C^{n\times n} (s,p)$ be a rational  matrix valued  function. Then the following conditions are equivalent:
\begin{enumerate}[\rm (i)]
\item $G$ does not have any singularities in $\mathbb G$ and the matrix $G(s,p)$ has nonnegative real part for any $(s,p) \in \mathbb{G}$;
\item $G$ admits a realization of the form
\begin{equation}\label{eq:gsp}
G(s,p)= D+C \bigl[I-\frac{s}{2}A -\frac14 (s^2-4p)PA(I-\frac{s}{2}A)^{-1}PA\bigr]^{-1}(\frac{s}{2}+\frac14(s^2-4p)PA(I-\frac s2 A)^{-1}P)B.
\end{equation}
where the block matrix $\matr{A&B\\C&D} \in\Cmats{(k+n)}{(k+n)}$ is such that
\begin{equation}\label{eq:c6a}
\matr{
I - A^* A & -A^* B+ C^* \\ -B^*A+ C & D+D^*- B^*B
} \geq 0
\end{equation}
and $P=P^*\in \sqmat{n}$ satisfies $-I\le P\le I$.
\end{enumerate}
Additionally, if $G\in {\mathbb R}^{n\times n}(s,p)$, then the matrices $A,B,C,D$ and $P$ in {\rm (ii)} can be chosen as real.
\end{thm}

\begin{proof} 
(i)$\Rightarrow$(ii):
    Given $G(s,p)$ with nonnegative real part, we let $F(z,\zeta) = G(z+\zeta, z\zeta).$ Then $F$ satisfies the conditions of Theorem \ref{thm:DDposRe} and thus we may write $F$ as in \eqref{eq:rDisc} and \eqref{eq:c6}. As $F(z,\zeta)=\frac12(f(z,\zeta)+f(\zeta,z))$ we may write 
   $$ \tiny{F(z,\zeta)=  D+\frac12\matr{ C& C}\left(I-(z\matr{ P_1&\\& P_2}
    +\zeta\matr{ P_2&\\& P_1})
    \matr{ A&\\& A}\right)^{-1}(z\matr{ P_1&\\& P_2}
    +\zeta\matr{ P_2&\\& P_1})\matr{ B\\ B}}.$$
    Inserting ${\tiny U=\begin{bmatrix} \frac{1}{\sqrt{2}} I & \frac{1}{\sqrt{2}} I \cr \frac{1}{\sqrt{2}} I & -\frac{1}{\sqrt{2}} I  \end{bmatrix}}$ in various places, we may rewrite this as
    $$F(z,\zeta)=  D+\frac12\matr{ \sqrt{2}C & 0 }\matr{I-\frac12 (z+\zeta) ( P_1+  P_2) A&-\frac12 (z-\zeta) ( P_1-  P_2)A\\-\frac12(z-\zeta) ( P_1-  P_2)A&I-\frac12(z+\zeta) ( P_1+  P_2) A}
    ^{-1}\times $$ $$\matr{\frac12 (z+\zeta) ( P_1+  P_2) &\frac12 (z-\zeta) ( P_1-  P_2)\\\frac12(z-\zeta) ( P_1-  P_2)&\frac12(z+\zeta) ( P_1+  P_2) }\matr{ \sqrt{2}B\\0}.$$
    Using that $P_1+P_2=I$, denoting $P=P_1-P_2$ and 
putting $s=z+\zeta, p=z\zeta$, and observing that $(z-\zeta)^2= s^2-4p$. As $ P_1+  P_2 >0$, we may rewrite the above as \eqref{eq:gsp}.

(ii)$\Rightarrow$(i): One can reverse the above arguments, and obtain that $F(z,\zeta)$ satisfies the conditions of Theorem \ref{thm:DDposRe}. 

The case $G\in {\mathbb R}^{n\times n}(s,p)$ follows from the corresponding part of Theorem~\ref{thm:DDposRe} by observing that all the steps above can be done over ${\mathbb R}$.

\end{proof}

\section{Determinantal representations}\label{sec:det}

Determinantal representations for polynomials without zeros in the bidisk were established in \cite[Theorem 2.1]{GKVW}, where it is shown that a polynomial $p(z_1, z_2)$ of bidegree $(n_1, n_2)$ with $p(0,0)=1$ has no zeros in $\mathbb{D}^2$ if and only if it can be written as $p(z_1, z_2)=\det(I-KZ)$, where $Z$ is an $(n_1+n_2) \times (n_1+n_2)$ diagonal matrix with $z_1, z_2$ on the diagional and $K$ is a contraction. The analogous representation for the bihalfplane, of the form $p(z_1,z_2)=c\det(A+z_1B_1+z_2B_2)$ with $\operatorname{Im} A \geq 0$, $B_1, B_2 \geq 0$ and $B_1+B_2=I$ was obtained in \cite[Theorem 3.2]{knese2019}. Symmetric polynomials without zeros in $\mathbb{D}^2$ and polynomials without zeros in the symmetrized bidisk $\mathbb{G}$ were treated in \cite{BW}, obtaining in both cases contractive determinantal representations. In this section we obtain the bihalfplane analogs of these results, namely determinantal representations for symmetric polynomials without zeros in $\RHPtwo$ and for polynomials without zeros in $\mathbb{GH}$.

\begin{thm}\label{phasdetrep}
Let $p(z_1,z_2)$ be a symmetric polynomial of degree $(n, n)$ without roots in $\RHPtwo$. Then there exist a constant $c\neq 0$, complex matrices $\alpha_0, \alpha_1, \alpha_2$, with $\alpha_1,\alpha_2$ positive semidefinite so that $\alpha_1+ \alpha_2 =I $, and 
$ {\rm Re}\ \alpha_0 = \frac{1}{2}(\alpha_0+\alpha_0^*) \ge 0$, an integer $0\le l \le n$, and a unitary involution $U$ so that 
$$ (z_1+1)^l(z_2+1)^l p(z_1, z_2)  = c\det( \alpha_0 + z_1\alpha_1 + z_2 \alpha_2 ) $$
and $U\alpha_0U=\alpha_0, U\alpha_1U=\alpha_2$.

In the case that $p(z_1,z_2)\in{\mathbb R}[z_1,z_2]$, the matrices $\alpha_0, \alpha_1$ and $U$ can be chosen to be real.
\end{thm}

We first need a lemma. We let $\sigma (B)$ denote the set of eigenvalues of $B$.
\begin{lemat}\label{lemma}
  Given a contraction 
  \begin{equation}\label{A}\begin{bmatrix} A_1 & A_2\cr A_2 & A_1\end{bmatrix},\end{equation}
  there exists an orthogonal projection
  $P={\tiny \begin{bmatrix} P_1 & P_2\cr P_2 & P_1\end{bmatrix}}$ and a contraction $B={\tiny \begin{bmatrix} B_1 & B_2\cr B_2 & B_1\end{bmatrix}}$ with $1\not\in \sigma(B)$, so that $PB=0=BP$ and
  $$\begin{bmatrix} A_1 & A_2\cr A_2 & A_1\end{bmatrix}= \begin{bmatrix} P_1 & P_2\cr P_2 & P_1\end{bmatrix} + \begin{bmatrix} B_1 & B_2\cr B_2 & B_1\end{bmatrix} .$$

  In case \eqref{A} is a real matrix, $B$ and $P$ can be chosen to be real as well.
\end{lemat}

\begin{proof}
    Using unitary similarity with the unitary Hermitian matrix $\Phi=\frac{1}{\sqrt{2}}{\tiny \begin{bmatrix} I & I \cr I & -I  \end{bmatrix}}$
we convert \eqref{A} to $(A_1+A_2) \oplus (A_1-A_2)$.

As $A_1 \pm A_2$ are contractions, their eigenspaces at eigenvalue 1 are reducing subspaces, and thus we can find unitary matrices $U_\pm$ so that 
$$U_\pm^* (A_1\pm A_2) U_\pm = \begin{bmatrix} I_{m_\pm} & 0 \cr 0 & K_\pm \end{bmatrix},$$
where $K_\pm$ are contractions without 1 as an eigenvalue. Decomposing $U_\pm = \begin{bmatrix} V_{\pm} & W_{\pm} \end{bmatrix}$, where $V_{\pm}$ has $m_\pm$ columns,
we get that 
$$ A_1\pm A_2 = U_\pm \begin{bmatrix} I_{m_\pm} & 0 \cr 0 & K_\pm \end{bmatrix} U_\pm^* = V_\pm V_\pm^* + W_\pm K_\pm W_\pm^*. $$
Note that $W_\pm K_\pm W_\pm^*$  are contractions without 1 as eigenvalue, $V_\pm V_\pm^*$ are orthogonal projections and $V_\pm V_\pm^*W_\pm K_\pm W_\pm^*=0, $ since $V_\pm^*W_\pm =0$, due to $U_\pm$ being unitary. These properties persist when taking direct sums and performing a unitary similarity using $\Phi$.
Thus
$$ \begin{bmatrix} A_1 & A_2\cr A_2 & A_1\end{bmatrix} = \Phi \begin{bmatrix} A_1 + A_2 & \cr 0 & A_1 - A_2\end{bmatrix} \Phi^* = \frac12 \begin{bmatrix} V_+V_+^*+ V_- V_-^* & V_+V_+^*- V_- V_-^* \cr V_+V_+^* - V_- V_-^* & V_+V_+^*+ V_- V_-^*\end{bmatrix} $$ $$+ \frac12 \begin{bmatrix} W_+K_+W_+^*+ W_- K_-W_-^* & W_+K_+W_+^*- W_- K_-W_-^* \cr W_+K_+W_+^*- W_- K_-W_-^* & W_+K_+W_+^*+ W_- K_-W_-^*\end{bmatrix}=: P+B$$
gives the desired sum decomposition. 

It is easy to check that all the steps above can be done within the reals when \eqref{A} is real, yielding the last statement of the lemma.\end{proof}

Let $\DD = \{ z \in \CC : |z|<1\}$ and ${\mathbb T} = \{ z \in \CC : |z|=1\}$.
Recall that the map
$$ \phi(z) =  \frac{1+z}{1-z} $$ maps the unit disk conformally onto the right half plane sending ${\mathbb T}$ to $i{\mathbb R} \cup \{ \infty \}$. The inverse is given by
$$ \phi^{-1}(w)= \frac{w-1}{w+1}. $$

\begin{proof}[Proof of Theorem \ref{phasdetrep}.]
The last statement of the theorem is proven by observing that if
the polynomial $p(z_1,z_2)$ has real coefficients,
all the steps below can be done over ${\mathbb R}$.

Put
\begin{equation}\label{ptilde}
\tilde{p} (z_1, z_2) = p\Big(\frac{1+z_1}{1-z_1},
\frac{1+z_2}{1-z_2}\Big)\prod_{j=1}^2(1-z_j)^{n}. 
\end{equation} 
Then, for $z_j\in\DD$, we have
$$ \tilde{p} (z_1, z_2) = p\Big( \frac{1+z_1}{1-z_1}, \frac{1+z_2}{1-z_2}\Big)\prod_{j=1}^2(1-z_j)^{n}  \neq 0. $$
 Thus $\tilde{p}$ is a symmetric polynomial without roots on ${\DD}^2$. We may now apply \cite[Theorem 2.2]{BW} (or Theorem \ref{detrepD2} in the real case) giving us the existence of a determinantal representation
\begin{equation}\label{dr}
\tilde{p}(z_1,z_2 ) = \tilde{p}(0,0) \det \left( I - \begin{bmatrix} A_1 & A_2\cr A_2 & A_1\end{bmatrix} \begin{bmatrix} z_1I_{n+l} & 0\cr 0 & z_2 I_{n+l} \end{bmatrix} \right)  \end{equation}
with $0\le l\le n$ and
\begin{equation}\label{contr} A=\begin{bmatrix} A_1 & A_2\cr A_2 & A_1\end{bmatrix} \end{equation}
a contraction. 
Going back to the original polynomial $p(z_1,z_2)$, we get 
for $z_1, z_2 \in \RHP$, 

\begin{align*} p(z_1,z_2)\prod_{j=1}^2 (z_j+1)^{l}&= \frac{1}{2^{2n}}\tilde{p} \Big(\frac{z_1-1}{z_1+1}, \frac{z_2-1}{z_2+1}\Big) \prod_{j=1}^2 (z_j+1)^{n+l} \\
&=  \frac{\tilde{p}(0,0)}{2^{2n}} \det\Big(\left(\oplus_{j=1}^2 (z_j+1)I_{n+l}\right)-A\left(\oplus_{j=1}^2 (z_j-1)I_{n+l}\right)\Big)\\
&= \frac{\tilde{p}(0,0)}{2^{2n}} \det( I+A + (I-A)(\oplus_{j=1}^2 z_jI_{n+l})), 
\end{align*}
Now we use the decomposition $A=P+B$ from Lemma \ref{lemma}. 
Recall that $P$ is an orthogonal projection and $P^\perp BP^\perp =B$, where  $P^\perp:=I-P$. 
Observe that with $Q(z)=\oplus_{j=1}^2 z_jI_{n+l} $ the matrix under the determinant has the following form with respect to the decomposition $\operatorname{Ran} P\oplus\ker P$:
 $$
 \begin{bmatrix}
     2I_{\text{Ran} P} &  0 \\
   *  & I_{\ker P}+B+ P^\perp Q(z)P^\perp - BQ(z)P^\perp   
    \end{bmatrix} 
 $$
$$
= \begin{bmatrix}
     2I_{\text{Ran} P} &  0 \\
   *  & I_{\ker P}+B+ (I-B)P^\perp Q(z)P^\perp   
    \end{bmatrix} 
$$

Therefore, 
\begin{align*}
 p(z_1,z_2)\prod_{j=1}^2 (z_j+1)^{l}&= \frac{\tilde{p}(0,0)}{2^{2n-{\rm rank} P}} \det\Big( I_{\ker P}+B+ (I-B)P^\perp Q(z)P^\perp \Big) \\ 
 &=  
\frac{\tilde{p}(0,0)}{2^{2n-{\rm rank} P}}\det(I-B) \det \Big((I-B)^{-1}(I+B) + P^\perp Q(z)P^\perp \Big).
\end{align*}
Notice that
$$
P^\perp Q(z)P^\perp=
\begin{bmatrix}
   I_{n+l}-P_1 & -P_2\\ -P_2 & I_{n+l}-P_1 
\end{bmatrix} 
\begin{bmatrix}
  z_1 I_{n+l} & 0\\ 0 & z_2I_{n+l} 
\end{bmatrix}
\begin{bmatrix}
   I_{n+l}-P_1 & -P_2\\ -P_2 & I_{n+l}-P_1 
\end{bmatrix}
$$ $$
=\begin{bmatrix}
z_1 (I_{n+\ell}-P_1)^2 + z_2 P_2^2 &
- z_1 (I_{n+\ell}-P_1)P_2 - z_2 P_2 (I_{n+\ell}-P_1)\\[6pt]
- z_1 P_2 (I_{n+\ell}-P_1) - z_2 (I_{n+\ell}-P_1)P_2 &
z_1 P_2^2 + z_2 (I_{n+\ell}-P_1)^2
\end{bmatrix}
$$
\[
=
z_1
\begin{bmatrix}
 (I_{n+\ell}-P_1)^2 & -(I_{n+\ell}-P_1)P_2\\
 -P_2 (I_{n+\ell}-P_1) & P_2^2
\end{bmatrix}
+
z_2
\begin{bmatrix}
 P_2^2 & -P_2 (I_{n+\ell}-P_1)\\
 -(I_{n+\ell}-P_1)P_2 & (I_{n+\ell}-P_1)^2
\end{bmatrix}.
\]
\[
=: z_1\alpha_1+ z_2\alpha_2
\]
Let also $\alpha_0 := (I-B)^{-1}(I+B)$,
so that the desired determinantal representation holds.
Observe that with 
$U={\tiny \begin{bmatrix} 0 & I \cr I & 0 \end{bmatrix}}$ one has  $U\alpha_0 U = \alpha_0$  to $UB=BU=BU^*$ and 
$U\alpha_1U=\alpha_2$.
\end{proof}

\begin{thm} \label{symm-bh-dr}
    Let $g(s,p)$ be a polynomial of degree $n$. Then $g(s,p)$ is without roots in $\mathbb{GH}$ if and only if there exists an integer $l \geq 0$ and a contraction $A = \begin{bmatrix} A_1 & A_2 \\ A_2 & A_1 \end{bmatrix}$ such that:
    $$(s+p+1)^l g(s,p) = c \det \left ( (I+2A_1+K) +s(I-K)+p(I-2A_1+K) \right),$$
    where $c$ is a non-zero constant and $K=(A_1+A_2)(A_1-A_2)$.

    In the case that $g(s,p)\in{\mathbb R}[s,p]$, the constant $c$ and the matrices $A_1, A_2$ and $K$ can be chosen to be real.
\end{thm}

\begin{proof}
The last statement of the theorem is proven by observing that if the polynomial $g(s,p)$ has real coefficients, all the steps below can be done over ${\mathbb R}$.

    The Cayley transform $w(z) = \frac{z-1}{z+1}$ and its inverse $z(w) = \frac{1+w}{1-w}$ provide a conformal bijection between the open right half-plane $\RHP$ and the open unit disk $\mathbb{D}$. 
    Let $(s, p)$ be the symmetrized coordinates on $\mathbb{GH}$ (where $s=z+\zeta$, $p=z\zeta$) and let $(\sigma, \pi)$ be the symmetrized coordinates on $\mathbb{G}$ (where $\sigma=w(z)+w(\zeta)$, $\pi=w(z)w(\zeta)$). By substituting the inverse transform $z(w)$ into $s$ and $p$, we express the bihalfplane variables in terms of the bidisk variables:
    $$s(\sigma, \pi) = \frac{1+w_1}{1-w_1} + \frac{1+w_2}{1-w_2} = \frac{2 - 2w_1 w_2}{1 - (w_1+w_2) + w_1 w_2} = \frac{2 - 2\pi}{1 - \sigma + \pi}$$
    $$p(\sigma, \pi) = \left( \frac{1+w_1}{1-w_1} \right) \left( \frac{1+w_2}{1-w_2} \right) = \frac{1 + (w_1+w_2) + w_1 w_2}{1 - (w_1+w_2) + w_1 w_2} = \frac{1 + \sigma + \pi}{1 - \sigma + \pi}$$
    Conversely, substituting the forward transform $w(z)$ into $\sigma$ and $\pi$ yields :
    $$\sigma(s, p) = \frac{z-1}{z+1} + \frac{\zeta-1}{\zeta+1} = \frac{2z\zeta - 2}{z\zeta + z + \zeta + 1} = \frac{2p-2}{s+p+1}$$
    $$\pi(s, p) = \left( \frac{z-1}{z+1} \right) \left( \frac{\zeta-1}{\zeta+1} \right) = \frac{z\zeta - (z+\zeta) + 1}{z\zeta + z + \zeta + 1} = \frac{p-s+1}{s+p+1}$$
    Notice from this that $1 - \sigma + \pi = \frac{4}{s+p+1}$.
    Suppose $g(s,p)$ is a polynomial of degree $n$ with no roots in $\mathbb{GH}$. Then the function
    $$f(\sigma, \pi) = (1 - \sigma + \pi)^n g(s(\sigma, \pi), p(\sigma, \pi))$$ is a polynomial without roots on the symmetrized bidisk $\mathbb{G}$.
    Applying \cite[Theorem 3.3]{BW} (or Theorem \ref{detrepG} in the real case) there exists a strict contraction $A = \begin{bmatrix} A_1 & A_2 \\ A_2 & A_1 \end{bmatrix}$, where  $A_1, A_2$ are $m \times m$ matrices with $m \geq n$ such that:
    $$f(\sigma, \pi) = c_0 \det \big( I_m - \sigma A_1 + \pi K \big),$$ for some constant $c_0$, where $K = (A_1+A_2)(A_1-A_2)$.

     Using $1 - \sigma + \pi = \frac{4}{s+p+1}$ we obtain:
    $$\left( \frac{4}{s+p+1} \right)^n g(s,p) = c_0 \det \left( I_m - \left( \frac{2p-2}{s+p+1} \right) A_1 + \left( \frac{p-s+1}{s+p+1} \right) K \right)$$

    Multiplying both sides by $(s+p+1)^m$ we get
    $$(s+p+1)^l g(s,p) = c  \det \Big( (s+p+1)I_m - (2p-2)A_1 + (p-s+1)K \Big),$$
    for some constant $c$ and $l=m-n \geq 0$. Rearranging terms inside the determinant yields
    $$(s+p+1)^l g(s,p) = c \det \Big( (I + 2A_1 + K) + s(I - K) + p(I - 2A_1 + K) \Big)$$
\end{proof}

\section*{Note in memory}
The function classes considered here belong to a broader body of work on reproducing kernel Hilbert spaces theory, to which Franciszek Hugon Szafraniec contributed: on the reproducing kernel Hilbert space and its multiplication operators \cite{Szafraniec2000}; on multipliers, subnormality and noncommutative complex analysis \cite{Szafraniec2003}; on Murphy's theory of positive definite kernels and Hilbert $C^*$-modules \cite{Szafraniec2010}; and, earlier, on multivariable holomorphic interpolation \cite{Szafraniec1986}. This consistent long-lasting line of research, as well the educational talents of  Franciszek Hugon Szafraniec, has  led to a meeting of three of his successors Radomił Baran, Piotr Pikul and Micha\l{} Wojtylak,  with Hugo Woerdeman. We record these connections here in the memory of the great professor, teacher and colleague.

\bibliographystyle{plain}

    \bibliography{bibio.bib}

\begin{thebibliography}{10}

\bibitem{AY}
Jim Agler and N.~J. Young.
\newblock Realization of functions on the symmetrized bidisc.
\newblock {\em J. Math. Anal. Appl.}, 453(1):227--240, 2017.

\bibitem{AIL}
A.~C. Antoulas, A.~C. Ionita, and S.~Lefteriu.
\newblock On two-variable rational interpolation.
\newblock {\em Linear Algebra Appl.}, 436(8):2889--2915, 2012.

\bibitem{AGP}
Athanasios~C. Antoulas, Ion~Victor Gosea, and Charles Poussot-Vassal.
\newblock On the {L}oewner framework, the {K}olmogorov superposition theorem,
  and the curse of dimensionality, 2025.

\bibitem{Arov}
D.~Z. Arov.
\newblock Passive linear steady-state dynamical systems.
\newblock {\em Sibirsk. Mat. Zh.}, 20(2):211--228, 457, 1979.

\bibitem{BGR}
Joseph~A. Ball, Israel Gohberg, and Leiba Rodman.
\newblock {\em Interpolation of rational matrix functions}, volume~45 of {\em
  Operator Theory: Advances and Applications}.
\newblock Birkh\"auser Verlag, Basel, 1990.

\bibitem{ball2015schur}
Joseph~A Ball and Dmitry~S Kaliuzhnyi-Verbovetskyi.
\newblock Schur--{A}gler and {H}erglotz--{A}gler classes of functions:
  positive-kernel decompositions and transfer-function realizations.
\newblock {\em Advances in Mathematics}, 280:121--187, 2015.

\bibitem{baran2026contractive}
Radomi{\l} Baran, Piotr Pikul, Hugo~J Woerdeman, and Micha{\l} Wojtylak.
\newblock Contractive realization theory for the annulus and other
  intersections of discs on the {R}iemann sphere.
\newblock {\em Journal of Functional Analysis}, 290:111346, 2026.

\bibitem{BW}
Radomi{\l} Baran and Hugo~J. Woerdeman.
\newblock Symmetric {S}chur-class functions on the bidisk and {S}chur-class
  functions on the symmetrized bidisk.
\newblock {\em Integral Equations Operator Theory}, 2026.
\newblock To appear.

\bibitem{BS}
Tirthankar Bhattacharyya and Haripada Sau.
\newblock Holomorphic functions on the symmetrized bidisk--{R}ealization,
  interpolation and extension.
\newblock {\em J. Funct. Anal.}, 274(2):504--524, 2018.

\bibitem{dBR1}
Louis de~Branges and James Rovnyak.
\newblock Canonical models in quantum scattering theory.
\newblock In {\em Perturbation {T}heory and its {A}pplications in {Q}uantum
  {M}echanics ({P}roc. {A}dv. {S}em. {M}ath. {R}es. {C}enter, {U}.{S}. {A}rmy,
  {T}heoret. {C}hem. {I}nst., {U}niv. of {W}isconsin, {M}adison, {W}is.,
  1965)}, pages 295--392. Wiley, New York-London-Sydney, 1966.

\bibitem{dBR2}
Louis de~Branges and James Rovnyak.
\newblock {\em Square summable power series}.
\newblock Holt, Rinehart and Winston, New York-Toronto-London, 1966.

\bibitem{Dritschel}
Michael~A. Dritschel.
\newblock Factoring non-negative operator valued trigonometric polynomials in
  two variables.
\newblock {\em Math. Ann.}, 391(1):515--537, 2025.

\bibitem{dritschel2026}
Michael~A. Dritschel, Igor Klep, Scott McCullough, and Jurij Volčič.
\newblock Addendum to "factoring non-negative operator valued trigonometric
  polynomials in two variables", 2026.
\newblock arXiv 2608.23073.

\bibitem{GKVVW}
A.~Grinshpan, D.~S. Kaliuzhnyi-Verbovetskyi, V.~Vinnikov, and H.~J. Woerdeman.
\newblock Stable and real zero polynomials in two variables.
\newblock {\em Multidimensional Systems and Signal Processing}, 27(1):1--26,
  2016.

\bibitem{grinshpan2016}
Anatolii Grinshpan, Dmitry~S. Kaliuzhnyi-Verbovetskyi, Victor Vinnikov, and
  Hugo~J. Woerdeman.
\newblock Matrix-valued {H}ermitian {P}ositivstellensatz, lurking contractions,
  and contractive determinantal representations of stable polynomials.
\newblock In {\em Operator theory, function spaces, and applications}, volume
  255 of {\em Oper. Theory Adv. Appl.}, pages 123--136. Birkh\"auser/Springer,
  Cham, 2016.

\bibitem{GKVW}
G.~Grinshpan, D.~S. Kaliuzhnyi-Verbovetskyi, V.~Vinnikov, and H.~J. Woerdeman.
\newblock Stable and real-zero polynomials in two variables.
\newblock {\em Multidimensional Systems and Signal Processing}, 27:1--26, 2016.

\bibitem{Gut1946}
L.~Guttman.
\newblock Enlargement methods for computing the inverse matrix.
\newblock {\em Annals of Mathematical Statistics}, 17(3):336--343, 1946.

\bibitem{HW}
Y.~Hachez and H.~J. Woerdeman.
\newblock The {F}ischer-{F}robenius transformation and outer factorization.
\newblock In {\em Operator Theory, Structured Matrices, and Dilations},
  volume~7 of {\em Theta Series in Advanced Mathematics}, pages 181--203.
  Theta, Bucharest, 2007.

\bibitem{knese2019}
G.~Knese.
\newblock Global bounds on stable polynomials.
\newblock {\em Complex Analysis and Operator Theory}, 13:1895--1915, 2019.

\bibitem{knese2021}
G.~Knese.
\newblock Kummert’s approach to realization on the bidisk.
\newblock {\em Indiana University Mathematics Journal}, 70(6):2369--2403, 2021.

\bibitem{Kummert}
Anton Kummert.
\newblock Synthesis of two-dimensional lossless {$m$}-ports with prescribed
  scattering matrix.
\newblock {\em Circuits Systems Signal Process.}, 8(1):97--119, 1989.

\bibitem{LT}
P.~Lancaster and M.~Tismenetsky.
\newblock {\em The Theory of Matrices}.
\newblock Computer Science and Applied Mathematics. Academic Press, Orlando,
  FL, 2 edition, 1985.

\bibitem{Szafraniec1986}
Franciszek~Hugon Szafraniec.
\newblock On bounded holomorphic interpolation in several variables.
\newblock {\em Monatsh. Math.}, 101(1):59--66, 1986.

\bibitem{Szafraniec2000}
Franciszek~Hugon Szafraniec.
\newblock The reproducing kernel {H}ilbert space and its multiplication
  operators.
\newblock In {\em Complex analysis and related topics ({C}uernavaca, 1996)},
  volume 114 of {\em Oper. Theory Adv. Appl.}, pages 253--263. Birkh\"auser,
  Basel, 2000.

\bibitem{Szafraniec2003}
Franciszek~Hugon Szafraniec.
\newblock Multipliers in the reproducing kernel {H}ilbert space, subnormality
  and noncommutative complex analysis.
\newblock In {\em Reproducing kernel spaces and applications}, volume 143 of
  {\em Oper. Theory Adv. Appl.}, pages 313--331. Birkh\"auser, Basel, 2003.

\bibitem{Szafraniec2010}
Franciszek~Hugon Szafraniec.
\newblock Murphy's positive definite kernels and {H}ilbert {$C^*$}-modules
  reorganized.
\newblock In {\em Noncommutative harmonic analysis with applications to
  probability {II}}, volume~89 of {\em Banach Center Publ.}, pages 275--295.
  Polish Acad. Sci. Inst. Math., Warsaw, 2010.

\bibitem{TGN}
P.~Triverio, S.~Grivet-Talocia, and M.~S. Nakhla.
\newblock A parametrized macromodeling strategy with uniform stability test.
\newblock {\em IEEE Transactions on Advanced Packaging}, 32(1):205--215, 2009.

\end{thebibliography}

\section*{Appendix}\renewcommand{\thesection}{A}

In this appendix we derive some real versions of results that are known over the complexes. 

\subsection*{Real Version of Arov's Contractive Realization}\label{Arov}
In this subsection we will prove the following real version of a
classical result due to Arov \cite{Arov}. 

\begin{thm}
Let $F(z) \in \mathbb{R}^{k\times l}(z)$ be a real rational matrix function that takes on contractive values for $z\in {\mathbb D}=\{ z \in{\mathbb C} : |z|<1\}$. Then $F(z)$ has a real contractive finite dimensional realization; that is, there exists a real contractive block matrix ${\scriptscriptstyle \begin{bmatrix} A & B \\ C & D \end{bmatrix}}\in{\mathbb R}^{(d{+}k)\times(d{+}l)}$ so that 
\begin{equation}\label{real} F(z) = D +zC (I -zA)^{-1} B   ,\quad z\in {\mathbb D}. \end{equation} This realization can be chosen to be minimal, i.e., $$
\bigcap_{j=0}^\infty \ker CA^{d-1} = \{0\}, \quad \bigvee_{j=0}^{d-1} \operatorname{ Ran} A^jB = {\mathbb R}^d. $$
\end{thm}

\begin{proof} Let $F(z) = D+zC(I-zA)^{-1}B$ be a minimal realization with real matrices $A,B,C,D$. Note that, by minimality, all the eigenvalues of $A$ lie inside ${\mathbb D}$. By Arov's result we may find a, possibly complex $S$, so that \begin{equation}\label{SABCD}
\begin{pmatrix}
S^{-1}AS & S^{-1}B \\
CS & D
\end{pmatrix}
\end{equation} is a contraction. Letting $T=SS^*$, we obtain that \begin{align}
\begin{bmatrix} T & \\ & I \end{bmatrix} - \begin{bmatrix}A&B\\C&D\end{bmatrix}  \begin{bmatrix}T & \\ & I \end{bmatrix} \begin{bmatrix}A^*&C^*\\B^*&D^*\end{bmatrix} \ge 0 .  \label{Tstatement}
\end{align} 
Using that $A,B,C,D$ are real, it is easy to see that the real matrix $\hat{T}:=\frac12(T+\overline{T})$ also works. We thus obtain  $\hat{T} - A\hat{T}A^* \ge BB^*  $. As $\vee_{j=0}^{d-1}\ {\rm Ran} A^jB = {\mathbb R}^d$, we get that 
$$\hat{T} - A^d\hat{T}A^{d*} = \sum_{j=0}^{d-1} A^j(\hat{T} - A\hat{T}A^{*})A^{j*} \ge \sum_{j=0}^{d-1} A^jBB^*A^{j*} >0.$$
Using now the Stein inequality result (see, e.g., \cite[Theorem 2 in Section 13.2]{LT}) that the number of positive eigenvalues of $\hat{T}$ equals the number of eigenvalues of $A^d$ inside ${\mathbb D}$, we obtain that $\hat{T}>0$. 
Factor now $\hat T = \hat{S}\hat{S}^*$, with $\hat{S}$ real, and ${\tiny \begin{pmatrix}
\hat S^{-1}A\hat S & \hat S^{-1}B \\
C\hat S & D
\end{pmatrix}}$ gives us the desired real contraction.

\subsection*{Real versions of realization Theorems \ref{thm:kneseDD}, \ref{thm:contrRHP}, \ref{thm:posReRHP},  and \ref{thm:DDposRe}}

\begin{thm}\label{thm:4vreal}
Let $F(z_1,z_2)$ be a rational function. The statements of Theorems \ref{thm:kneseDD}, \ref{thm:contrRHP}, \ref{thm:posReRHP},  and \ref{thm:DDposRe}
hold when appending simultaneously:
\begin{enumerate}[]
\item $F$ being real rational matrix function  in statement {\rm (i);}
\item $A,B,C,D,P_1,P_2$ being real matrices in statement {\rm (ii)}.
\end{enumerate}
\end{thm}

\begin{proof}
It is obvious that if the matrices $A,B,C,D,P_1,P_2$ are real in (ii) then $F(\overline{z_1},\overline{z_2})=\overline{F(z_2,z_1)}$, the only problem in the proof is the converse, which has to be shown separately for each result.  For a complex matrix $X$ we will use in the proof the notation 
\begin{equation}\label{eq:notations2}
\bar X=[\bar x_{ij}]_{ij},\quad  
\Re X=\frac{X+ \bar X}2,\quad \Im X=\frac{X-\bar X}{2i},\quad X_{\mathbb R}={\small \matr{  \Re X & \Im X\\ -\Im X & \Re X }}.
\end{equation}

{\em Step 1. Real realization formula for  Theorems~\ref{thm:contrRHP} and \ref{thm:posReRHP}}. 
These two results  share the same realization
\begin{equation}\label{eq:step1}
F(z_1,z_2)= D+ C(z_1P_1+z_2P_2+A)^{-1}B.
\end{equation}
We show first that the matrices $A,B,C,D$, $P_1,P_2$ can be chosen as real, in subsequent steps we will complete the system matrix conditions. Take a complex realization \eqref{eq:step1} with complex matrices and assume that the additional condition $F(\overline{z_1},\overline{z_2})=\overline{F(z_2,z_1)}$ holds. 
 As $F(t,t)$ is real, $ D$ is real as well.
 For fixed $z_1,z_2\in\mathbb R$ we have 
\begin{equation*}
\begin{bmatrix} \Re F(z_1,z_2) & \Im F(z_1,z_2) \\ -\Im F(z_1,z_2) & \Re F(z_1,z_2) \end{bmatrix} = \begin{bmatrix}  D & 0 \\ 0 & D \end{bmatrix}
\end{equation*}  
\begin{equation*}
+ \begin{bmatrix} \Re C & \Im C   \\ -\Im C & \Re C \end{bmatrix}
\left ( z_1   \begin{bmatrix} \Re P_1 & \Im P_1   \\ -\Im P_1 & \Re P_1 \end{bmatrix} 
+ z_2\begin{bmatrix} \Re P_2 & \Im P_2   \\ -\Im P_2 & \Re P_2\end{bmatrix} 
+\begin{bmatrix} \Re A & \Im A   \\ -\Im A & \Re A \end{bmatrix} 
\right)^{-1}
\begin{bmatrix} \Re B & \Im B   \\ -\Im B & \Re B \end{bmatrix}.
\end{equation*}
As $\Im F(z_1,z_2)=0$ we have 
\begin{equation*}
F(z_1,z_2)=
\Re F(z_1,z_2) =  \begin{bmatrix} I & 0  \end{bmatrix}  \begin{bmatrix} \Re F(z_1,z_2) & \Im F(z_1,z_2) \\ -\Im F(z_1,z_2) & \Re F(z_1,z_2) \end{bmatrix} \begin{bmatrix} I \\ 0  \end{bmatrix} +D
\end{equation*}
\begin{equation*}
=\begin{bmatrix} \Re C & \Im C   \end{bmatrix}\left ( z_1   \begin{bmatrix} \Re P_1 & \Im P_1   \\ -\Im P_1 & \Re P_1 \end{bmatrix} 
+ z_2{ \begin{bmatrix} \Re P_2 & \Im P_2   \\ -\Im P_2 & \Re P_2\end{bmatrix} }
+\begin{bmatrix} \Re A & \Im A   \\ -\Im A & \Re A \end{bmatrix} 
\right)^{-1}
\begin{bmatrix} \Re B    \\ -\Im B  \end{bmatrix} +D.
\end{equation*}
As the set $(\mathbb R,\mathbb R)$ is a set of uniqueness in $\mathbb C^2$, the above formula holds for complex $z_1,z_2$ as well. 
Hence, we obtained  a real representation of $F(z_1,z_2)$ with 
\begin{align}
\tilde A&= \matr{ \Re A & \Im A   \\ -\Im A & \Re A } =A_{\mathbb R},\quad 
\tilde B= \matr{ \Re B    \\ -\Im B  },\quad
\tilde C= \matr{ \Re C & \Im C },\quad \tilde D=D, \label{eq:tildas1}\\
\tilde P_j&
=\matr{\Re P_j & \Im P_j   \\ -\Im P_j & \Re P_j }
=(P_j)_{\mathbb R}, \quad j=1,2.\label{eq:tildas2}
\end{align}
It is clear that $\tilde P_1,\tilde P_2$ are positive semidefinite matrices, with their sum equal to identity. 
\medskip

\emph{Step 2. Real realization formula for Theorems~\ref{thm:kneseDD}, and \ref{thm:DDposRe}}.
For the realization
\[
F(z_1,z_2)= D+ C(P_1z_1+P_2z_2)\big(I - A(P_1z_1+P_2z_2) \big)^{-1}B
\]
 the proof goes in a similar way, the matrices $\tilde A$, $\tilde B$, $\tilde C$, $\tilde D$, and $\tilde P_j$ ($j=1,2$) are defined again according to \eqref{eq:tildas1}, and \eqref{eq:tildas2}.
Additionally, note that in the case of Theorem~\ref{thm:kneseDD} the matrices $\tilde P_j$  ($j=1,2$) remain complementary orthogonal projections.\medskip

\emph{Step 3. System matrix condition in Theorem~\ref{thm:kneseDD}}
Let $T={\scriptscriptstyle \matr{A & B\\ C & D}}$ be the system matrix for the complex realization \eqref{eq:step1}. We have
$$
\norm{T_{\mathbb R}}=\norm T\leq 1,
$$
where the first norm is real and the second complex. Note that the system matrix $\tilde T$ for the real realization \eqref{eq:tildas1}, \eqref{eq:tildas2} is a submatrix of $T_{\mathbb R}$, hence it is contractive as well.\medskip

\emph{Step 4. System matrix condition in Theorem~\ref{thm:posReRHP}.}
Note that the complex realization \eqref{eq:step1} satisfies the KYP inequality \eqref{eq:kyp0}, which can be rewritten as 
$$
S+S^*\geq0 ,\quad S={\small \matr{A & -B\\ C & D}}.
$$
Therefore, we have as well that
$$
0\leq S_{\mathbb R} + (S^*)_{\mathbb R} = S_{\mathbb R} + (S_{\mathbb R})^\top.
$$
Observe now that for $\tilde S={\scriptscriptstyle\matr{\tilde A & -\tilde B\\ \tilde C & \tilde D}}$ the matrix $\tilde S+\tilde S^\top$ is a submatrix of $S_{\mathbb R} + (S_{\mathbb R})^\top$ and the statement holds.
\medskip

\emph{Step 5. System matrix condition in Theorems~\ref{thm:contrRHP}, and \ref{thm:DDposRe}}
We prove only the former one, as the latter is similar.
Let $M$ denote the matrix in \eqref{eq:c1}.
We have that after a suitable permutation of rows and collumns the matrix $M_{\mathbb R}$ has the form
$$
\tilde M_{\mathbb R}=\matr{ A_{\mathbb R}+(A_{\mathbb R})^\top  - (C_{\mathbb R})^\top C_{\mathbb R} & (C_{\mathbb R})^\top D_{\mathbb R} + B_{\mathbb R}   \\ (D_{\mathbb R})^\top C_{\mathbb R}+    (B_{\mathbb R})^\top & I- (D_{\mathbb R})^\top D_{\mathbb R}) }.
$$
Observe that 
\begin{align*}
     (C_{\mathbb R})^\top C_{\mathbb R} &= \matr{ \Re C^\top \Re C +\Im C^\top \Im C&   \Re C^\top \Im C -\Im C^\top \Re C \\ \Im C^\top \Re C- \Re C^\top \Im C  &  \Re C^\top \Re C +\Im C^\top \Im C   } \\
     &= 
 \matr{\Re C^\top \\ \Im C^\top} \matr{\Re C & \Im C} +  \matr{\Im C^\top \\ -\Re C^\top} \matr{\Im C & -\Re C}\\
 &\geq \tilde C^\top \tilde C.
\end{align*}
Hence,
$$
0\leq\tilde M_{\mathbb R}\leq 
\matr{\tilde A+\tilde A^\top -\tilde C^\top \tilde C  & (C_{\mathbb R})^\top D_{\mathbb R} + B_{\mathbb R}   \\ (D_{\mathbb R})^\top C_{\mathbb R}+    (B_{\mathbb R})^\top & I- (D_{\mathbb R})^\top D_{\mathbb R}) }=:M'.
$$
It is now enough to see that the matrix in \eqref{eq:c1} for matrices $\tilde A,\tilde B,\tilde C,\tilde D$ is a submatrix of $M'$.
\end{proof}

\subsection*{Real version of Fej\'er-Riesz factorization}

\begin{thm} Suppose that $Q(z) =\sum_{k=-n}^n Q_j z^j$, with $Q_j=Q_{-j}^* \in {\mathbb R}^{m\times m}$, $j=0,\ldots , n$, and $Q(z) > 0$, $z\in{\mathbb T}$. Then there exists $P(z) =\sum_{j=0}^n P_j z^j$, with $P_j \in {\mathbb R}^{m\times m}$, $j=0,\ldots , n$, so that $Q(z)=P(z)^*P(z)$, $z\in {\mathbb T}$, and $\det P(z) \neq 0$, $z \in\overline{\mathbb D}$ (thus, $P(z)$ is outer). \end{thm}

\begin{proof} Following \cite[Proposition 4.6]{HW} one can find an outer factorization of $Q(z)$ by considering the convex compact set
$$ \{ X=(X_{ij})_{i,j=0}^n : X\ge 0, \sum_{j=0}^{n-k} X_{k+j,k}=Q_k, k=0,\ldots ,n \}, $$
and choosing the unique $X_{\rm opt}$ in this set that maximizes ${\rm trace} X_{nn}$. Then ${\rm rank} X_{\rm opt}= {\rm rank} Q_0 \le m$, and thus we may factor 
$$ X_{\rm opt} = \begin{pmatrix} P_n^* \cr \vdots \cr P_0^* \end{pmatrix} \begin{pmatrix} P_n & \dots & P_0 \end{pmatrix} ,
$$ with $P_j$ of size ${m\times m}$, $j=0,\ldots , n$. This gives the desired outer factorization. When the initial data is real, this process can be fully performed within the set of real matrices, yielding the real version of Fej\'er-Riesz factorization. 
\end{proof}

\subsection*{Real determinantal representation}

\begin{thm}\label{main}
Let $p(z_1, z_2)$, with $p(0, 0)=1$, be a non-constant 
bivariate polynomial with real coefficients. Then $p$ admits a representation 
\begin{equation}\label{eq:repr} p(z_1,z_2)=\det (I_{|n|} - K Z_n ),
\end{equation}
with $n=(n_1,n_2)=\deg p$ and $K\in {\mathbb R}^{|n|\times |n|}$ a contraction, where $|n|=n_1+n_2$ and
$Z_n=z_1I_{n_1}\oplus z_2I_{n_2}$. 
\end{thm}

The proof consist of observing that the key steps in the determinantal representation result \cite[Theorem 1.1]{GKVVW} can all be performed over the reals when the initial two variable polynomial $p(z_1,z_2)$ has real coefficients. These key steps are:

\begin{itemize}
\item First argue that it suffices to prove the result for a dense set of polynomials, which allows one to assume  certain generic conditions on the polynomial $p(z_1, z_2) $. The argument uses the compactness of the set of $n\times n$ contractions. 
\item Expand $p$ in the powers of $z_2$, $p(z_1, z_2) = p_0( z_1 ) + \cdots + p_{n_2}( z_1) z_2^{n_2}$, where $p_j(z_1) \in {\mathbb R}[z_1]$, $j=0,\ldots , n_2$,
and introduce the companion matrix
\begin{equation}\label{Cjdef2}   C (z_1) = \begin{bmatrix} \frac{p_1(z_1)}{p_0(z_1)} & 1 & & 0 \cr \vdots &
& \ddots &\cr -\frac{p_{n_2-1}(z_1)}{p_0(z_1)} & 0 & & 1\cr
- \frac{p_{n_2}(z_1)}{p_0(z_1)} & 0 & \cdots & 0 \end{bmatrix},  \end{equation} which has real coefficients.
\item Introduce the triangular Toeplitz matrices
\begin{equation*}   A(z_1) =
\begin{bmatrix} p_0 (z_1)& & \cr \vdots & \ddots & \cr p_{n_2-1} (z_1) &
\cdots & p_0(z_1) \end{bmatrix},\quad   B(z_1) = \begin{bmatrix} p_{n_2} (z_1) &
\cdots & p_{1} (z_1) \cr & \ddots & \vdots \cr & & p_{n_2} (z_1)
\end{bmatrix},  \end{equation*} and form the Bezoutian $Q(z_1) :=  
A(z_1)  {A}(1/\overline{z_1})^* -   B(1/\overline{z_1})^*  
B(z_1)$. Now $Q$ is a trigonometric polynomial with real matrix coefficients.

\item Observe that $Q(z_1)$ is positive definite for
every $z_1\in\mathbb T$ and finds its real Fej\'er-Riesz outer factorization, i.e., a $n_2\times n_2$ matrix-valued polynomial $P(z_1) = P_0 + \cdots
+ P_{n_1} z_1^{n_1}$ with real matrix coefficients such that the factorization $Q(z_1) =
P(z_1)^*P(z_1)$, $z_1 \in {\mathbb T}$, holds and $P(z_1)$ is
invertible for $z_1 \in\overline{\mathbb D}$.

\item Introduce the rational matrix-valued function
$$ M(z_1) := P(z_1)^{-1}  C(z_1) P(z_1),$$ which is analytic on
$\overline{\mathbb D}$. Then, $M$ is contractive and has real coefficients. Now apply the real version of Arov's theorem, yielding a contraction ${\tiny \begin{pmatrix} A & B \cr C & D \end{pmatrix}}$.

\item Finally, observe that
$$p(z_1, z_2)= \det\left(I_{|n|}-\begin{bmatrix} A & B\\ C & D\end{bmatrix}\begin{bmatrix}
z_1I_{n_1}& 0\\ 0 & z_2I_{n_2}\end{bmatrix}\right),$$ and we are done.

\end{itemize}
\end{proof}

As a consequence of Theorem \ref{main} we can now also state real versions of \cite[Theorem 2.2 and Theorem 3.3]{BW} and \cite[Theorem 3.2]{knese2019}

\begin{thm}\label{detrepG}
     Let $g(s,p)$ be a polynomial with real coefficients. Then $g(s,p)$ is without roots in ${\mathbb G}$ if and only if there exists a real contraction \begin{equation}\label{contr3} \begin{pmatrix} A_1 & A_2\cr A_2 & A_1\end{pmatrix} \end{equation} such that 
    \begin{equation}\label{gh} g(s,p)= g(0,0) \det (I -sA_1 + p(A_1+A_2)(A_1-A_2)).\end{equation}
    In addition, $g(s,p)$ is without roots in $\overline{\mathbb G}$ if and only if there exists a strict real contraction \eqref{contr3} so that 
    \eqref{gh} holds.
\end{thm}

\begin{proof} We repeat the proof of \cite[Theorem 3.3]{BW}, but now over the reals.

Let $g(s,p)\in{\mathbb R}[s,p]$ be without roots in ${\mathbb G}$. For ease of the presentation, let us assume that $g(0,0)=1$ (otherwise, consider $\frac{g(s,p)}{g(0,0)}$). Put $h(z,\zeta)=g(z+\zeta, z\zeta).$ Then $h(z,\zeta)=h(\zeta, z) \in{\mathbb R}[z,\zeta]$ is symmetric and without roots in the closed bidisk. Let $h(z,\zeta)$ be of degree $(n,n)$. Let $\sigma=\frac12(z+\zeta)$ and $\delta=\frac12 (z-\zeta)$. Then $h(z,\zeta)=h(\sigma+\delta,\sigma-\delta)=q(\sigma,\delta)$ for some real polynomial $q$. Since $h(z,\zeta)=h(\zeta,z)$ we get that $q(\sigma,\delta)=q(\sigma,-\delta)$. Thus $q$ is even in $\delta$, so $q$ is a polynomial in $\sigma$ and $\delta^2$. Moreover, $q$ as a polynomial in $\sigma$ and $\delta^2$ is of degree $(n,m)$, where $m=\lfloor \frac{n}{2} \rfloor$, and $q$ is without roots in ${\mathbb D}^2$. Applying Theorem \ref{main} to $q$ as a polynomial in $\sigma$ and $\delta^2$, we obtain that there exists a $(n+m)\times (n+m)$ real contraction
$$ \begin{pmatrix} K_{11} & K_{12}\cr K_{21} & K_{22}\end{pmatrix}$$ such that
$$
q(\sigma,\delta)= \det \left( I_{n+m} - 
\begin{pmatrix} K_{11} & K_{12}\cr K_{21} & K_{22}\end{pmatrix} \begin{pmatrix} \sigma I & 0\cr 0 & \delta^2 I_m\end{pmatrix} \right).
$$
Let $s=z+\zeta$ and $p=z\zeta$. Then $\sigma=\frac{s}{2}$ and $\delta^2=-p+\frac{s^2}{4}$ and since $g(s,p)=q(\sigma,\delta)$, we find that
$$
g(s,p) = \det \left( I - \begin{pmatrix} \frac12 K_{11} & -K_{12}\cr \frac12 K_{21} & -K_{22}\end{pmatrix} \begin{pmatrix} s I & 0\cr 0 & p I_m\end{pmatrix} - \frac{s^2}{4} \begin{pmatrix} 0 & K_{12}\cr 0 & K_{22}\end{pmatrix} \right).
$$ We may rewrite this, using Schur complements, as
$$ g(s,p) = \det \left( I - s \begin{pmatrix} \frac12 K_{11} & 0 & -\frac12 K_{12}\cr \frac12 K_{21} & 0 &  -\frac 12 K_{22}\cr 0 & -\frac12 I_m & 0 \end{pmatrix} +p \begin{pmatrix} 0 & K_{12} & 0\cr 0 & K_{22} & 0 \cr 0 & 0 & 0 \end{pmatrix} \right). $$
Let now $$ A_1= \begin{pmatrix} \frac12 K_{11} & 0 & -\frac12 K_{12}\cr \frac12 K_{21} & 0 &  -\frac 12 K_{22}\cr 0 & -\frac12 I_m & 0 \end{pmatrix} , A_2 = \begin{pmatrix} \frac12 K_{11} & 0 & -\frac12 K_{12}\cr \frac12 K_{21} & 0 &  -\frac 12 K_{22}\cr 0 & \frac12 I_m & 0 \end{pmatrix}.$$ Then \eqref{gh} follows. Note also that $A_1$ and $A_2$ are of size $n+2m \le 2n$. Finally, it is easy to check that $\eqref{contr3}$ is a contraction (as $\eqref{contr3}$ being a contraction is equivalent to $\| A_1 \pm A_2 \| \le 1$). 

If $g(s,p)$ is without roots in $\overline{\mathbb G}$, then there exists $R>1$ so that $g(Rs,R^2p)$ is without roots $(s,p)\in {\mathbb G}$. This yields that $g(Rs,R^2p)$ may be expressed as the righthand side of \eqref{gh}. But then $\frac{1}{R}A_1$ and $\frac{1}{R} A_2$ yields the desired strict contraction in the determinantal representation \eqref{gh} for $g(s,p)$. 
\end{proof}

\begin{thm}\label{detrepD2}
A scalar valued symmetric polynomial $g(z,\zeta) \in {\mathbb R}[z,\zeta]$ of degree $(n,n)$ has no roots in ${\mathbb D}^2$ if and only if $g$ has a determinantal representation
\begin{equation}\label{dr2} g(z,\zeta ) = g(0,0) \det \left( I - \begin{pmatrix} A_1 & A_2\cr A_2 & A_1\end{pmatrix} \begin{pmatrix} zI & 0\cr 0 & \zeta I \end{pmatrix} \right)  \end{equation}
with 
a real contraction 
\begin{equation} \label{contrDetBidiskReal}
    \begin{pmatrix} A_1 & A_2\cr A_2 & A_1\end{pmatrix}.
\end{equation} In addition, $g(z,\zeta)$ has no roots in $\overline{\mathbb D}^2$ if and only if $g$ has a determinantal representation \eqref{dr2} with \eqref{contrDetBidiskReal} a real strict contraction. In both cases the matrices $A_1$ and $A_2$ can be chosen to be of size at most $2n\times 2n$.  
\end{thm}

\begin{proof}
We repeat the proof of \cite[Theorem 2.2]{BW}, but now over the reals.

Let $g(z,\zeta)$ be a symmetric polynomial without roots in ${\mathbb D}^2$.
Let $s=z+\zeta$ and $d=z-\zeta$. Then $g(z,\zeta)=g(\frac{s+d}2,\frac{s-d}2)=h_0(s,d)$ for some polynomial $h_0$. Since $g(z,\zeta)=g(\zeta,z)$ we get that $h_0(s,d)=h_0(s,-d)$. Thus $h$ is even in $d$, so $h$ is a polynomial in $s$ and $d^2$. As $d^2= s^2 -4p$, where $p=z\zeta$, we get that $h_0$ is in fact a polynomial $h(s,p)$ in $s$ and $p$. Thus $g(z,\zeta) = h(s,p)$, and as $g$ is without roots in ${\mathbb D}^2$ the polynomial $h(s,p)\in {\mathbb R}[s,p]$ is without roots in ${\mathbb G}$. Now, Theorem \ref{detrepG} yields the existence of real matrices $A_1$ and $A_2$ of size $\ell \le 2n$ so that \eqref{contrDetBidiskReal} is a contraction and \eqref{gh} holds. We claim that the right hand side of \eqref{dr2} equals \eqref{gh}. For convenience we assume that $h(0,0)=g(0,0)=1$.

Let $z\neq 0 \neq \zeta$ and put $\hat{s}=\frac1z+\frac{1}{\zeta}, \hat{d} = \frac1z-\frac{1}{\zeta}.$ Then $\hat{s}=\frac{s}{p}, \hat{d}^2=\frac{s^2-4p}{p^2}$. Moreover, using ${\tiny U=\begin{bmatrix} \frac{1}{\sqrt{2}} I & \frac{1}{\sqrt{2}} I \cr \frac{1}{\sqrt{2}} I & -\frac{1}{\sqrt{2}} I  \end{bmatrix}}$,
$$ \det U\left( I - \begin{pmatrix} A_1 & A_2\cr A_2 & A_1\end{pmatrix} \begin{pmatrix} zI_\ell & 0\cr 0 & \zeta I_\ell \end{pmatrix} \right)U^* = $$ $$z^\ell \zeta^\ell \det \left( \frac12\begin{pmatrix} \hat{s}I & \hat{d}I \cr \hat{d}I & \hat{s} I\end{pmatrix} - \begin{pmatrix} A_1+A_2 & 0 \cr 0 & A_1-A_2 \end{pmatrix} \right) = $$
$$ z^\ell \zeta^\ell \det \left( (\frac{\hat{s}}{2}I_\ell -(A_1+A_2)) (\frac{\hat{s}}{2} I_\ell -(A_1-A_2)) - \frac{\hat{d}^2}{4}I_\ell \right) = $$
$$ p^\ell \det(\frac{1}{p}I_l - \frac{s}{p} A_1 +(A_1+A_2)(A_1-A_2)) = g(s,p).$$ This proves that \eqref{dr2} holds (where we observe that for $z=0$ or $\zeta=0$ the equality follows by continuity).

When $g(z,\zeta)$ has no roots in $\overline{\DD}^2$ one may find an $R>1$ so that $g(Rz,R\zeta)$ has not roots for $(z,\zeta) \in {\mathbb D}^2$. One can thus represent $g(Rz,R\zeta)$ as the right-hand-side of \eqref{dr2}. But then {\tiny $\frac{1}{R} \begin{pmatrix} A_1 & A_2\cr A_2 & A_1\end{pmatrix}$} is the desired real strict contraction.
\end{proof} 

\begin{thm}\label{real-Knese}
Let $p \in \mathbb{R}[z_1, z_2]$ be a real polynomial with bidegree $(n,m)$, without zeros in the right bihalfplane $\RHPtwo$. Then $p$ admits a determinantal representation 
\begin{equation}\label{eq:KneseBHR} p(z_1,z_2)=c\det (A+z_1B_1+z_2B_2),
\end{equation}
where $A, B_1, B_2$ are real matrices of size $\deg p \times \deg p$ such that $S_A \coloneqq \frac{1}{2}(A+A^T) \geq 0$, $B_1, B_2 \geq 0$ and $B_1+B_2=I$. 
\end{thm}
\begin{proof}
    Define $\phi(\zeta) = \frac{1+\zeta}{1-\zeta}$ and
\[
q(z_1,z_2) = p(\phi(z_1),\phi(z_2))\left(\frac{1-z_1}{2}\right)^n \left(\frac{1-z_2}{2}\right)^m.
\]
One can calculate that $\phi^{-1}(\zeta) = \frac{\zeta - 1}{\zeta +
  1}$ and
\[
p(z_1,z_2) = q(\phi^{-1}(z_1),\phi^{-1}(z_2)) (z_1+1)^n(z_2+1)^m.
\]
Then, $q$ is a real polynomial that has no zeros in $\mathbb{D}^2$ and so the conclusion
of Theorem \ref{main} holds. Hence, if
\[
P_1 = \begin{pmatrix} I & 0 \\ 0 & O_m \end{pmatrix} \qquad 
P_2 = \begin{pmatrix} O_n & 0 \\ 0 & I_m \end{pmatrix},
\]
converting \eqref{eq:repr} to a formula for $p$ yields a contractive $K \in \mathbb{R}^{(n+m) \times (n+m)}$ such that
\[
\begin{aligned}
p(z) &= c\det( (z_1+1)P_1 + (z_2+1)P_2 - K((z_1-1)P_1+(z_2-1)P_2))\\
&= c \det( (I-K)(z_1P_1+z_2P_2) + (I+K))
\end{aligned}
\]
Since $K$ is a contraction, the eigenspace corresponding to eigenvalue $1$ is reducing (if nontrivial). 
Thus, there exists a real orthogonal matrix $U$ such that
\[
K = U\begin{pmatrix} I & 0 \\ 0 & L\end{pmatrix} U^T
\]
where $L$ is a contractive $k\times k$ matrix for which $1$ is not an 
eigenvalue. Here $k$ is the codimension of the eigenspace of $K$ corresponding to eigenvalue $1$. 
Then,
\[
\begin{aligned}
p(z) &= c \det\left( \begin{pmatrix} 0 & 0 \\ 0 & I-L \end{pmatrix} U^T (z_1P_1+z_2P_2) U + 
\begin{pmatrix} 2I & 0 \\ 0 & I+L \end{pmatrix}\right)\\
&= c \det(I-L)
\det\left( \begin{pmatrix} 0 & 0 \\ 0 & I \end{pmatrix} U^T(z_1P_1+z_2P_2) U + \begin{pmatrix} 2I & 0 \\ 0 & A \end{pmatrix}\right)
\end{aligned}
\]
where $A = (I+L)(I-L)^{-1}$ is a real matrix. 

Let $B_j$ equal the bottom right $k\times k$ block of $U^TP_jU$. Clearly $B_1, B_2$ are real and positive semidefinite. Moreover,
\[
p(z) = c\det(I-L) 
\det\begin{pmatrix} 2I & 0 \\ * & A + z_1B_1 +z_2B_2 \end{pmatrix}
= c_0 \det(A+z_1B_1+z_2B_2)
\]
where $c_0$ is a new constant.  Since $U$ is real, $B_1, B_2$ are real and since $P_1+P_2 = I$, 
$B_1+B_2= I$.  Note that $p(t,t) = c_0 \det(A+tI)$ has
degree $k$ so that $k \leq \deg p$.  On the other hand, 
the determinantal formula for $p$ has total degree at most $k$, so that $\deg p \leq k$.  
Therefore the matrices in our formula have size matching
the total degree of $p$.
Finally, 
\[
S_A 
 = (I-L^T)^{-1}(I-L^TL)(I-L)^{-1} \geq 0.
\]%
\end{proof}

\end{document}